\documentclass{amsart}

\usepackage{amsmath}
\usepackage{amsthm}
\usepackage{amssymb}
\usepackage{braket}
\usepackage{mathtools}

\usepackage[all,cmtip,2cell]{xy}
\usepackage{soul}

\usepackage{tikz}
\usetikzlibrary{arrows,decorations.pathmorphing,backgrounds,positioning,fit,petri}

\usepackage[colorlinks=true, allcolors=blue]{hyperref}

\usepackage{color}

\newtheorem{thm}{Theorem}[section]
\newtheorem{theorem}[thm]{Theorem}

\newtheorem{lemma}[thm]{Lemma} 
\newtheorem{proposition}[thm]{Proposition}

\theoremstyle{definition}
\newtheorem{definition}[thm]{Definition}
\newtheorem{example}[thm]{Example}
\newtheorem{observation}[thm]{Observation}
\newtheorem{remark}[thm]{Remark}
\newtheorem{notation}[thm]{Notation}

\newcommand{\Sets}{\mathrm{Sets}}
\newcommand{\sSets}{\mathrm{sSets}}
\newcommand{\ssSets}{\mathrm{ssSets}}
\newcommand{\diGraphs}{\mathrm{diGraphs}}
\newcommand{\diGraphss}{\mathrm{diGraphs}_{\leq 1}}

\newcommand{\Graphs}{\mathrm{Graphs}}
\newcommand{\stGraphs}{\mathrm{Graphs}_{\leq 1}}

\newcommand{\schtwo}{\mathcal{C}_{\mathrm{Sch}}^2}
\newcommand{\mantwo}{\mathcal{C}_{\mathrm{Man}}^2}

\newcommand{\Vectk}{\mathrm{Vect}_\mathbb{K}}

\newcommand{\Hom}{\mathrm{Hom}}
\newcommand{\Fun}{\mathrm{Fun}}

\newcommand{\pr}{\mathrm{pr}}

\newcommand{\Spec}{\mathrm{Spec}}
\newcommand{\RngSpcs}{\mathrm{RngSpcs}}
\newcommand{\LRngSpcs}{\mathrm{LRngSpcs}}
\newcommand{\RSfin}{\mathrm{RngSpcs}_{\mathrm{Fin}}}
\newcommand{\Rings}{\mathrm{Rings}}
\newcommand{\Sch}{\mathrm{Sch}}

\newcommand{\qCoh}{\mathrm{qCoh}}
\newcommand{\cube}{\Xi}
\newcommand{\Top}{\mathrm{Top}}
\newcommand{\Man}{\mathrm{Man}}
\newcommand{\Specr}{\mathrm{Spec}_m}
\newcommand{\cCinf}{\cC^\infty}
\newcommand{\ralg}{\mathbb{R}\textnormal{-algebras}}

\newcommand{\colim}{\mathrm{colim}}
\newcommand{\G}{\boldsymbol{G}}

\newcommand{\cA}{\mathcal{A}}
\newcommand{\cB}{\mathcal{B}}
\newcommand{\cC}{\mathcal{C}}
\newcommand{\cS}{\mathcal{S}}
\newcommand{\cO}{\mathcal{O}}

\newcommand{\cF}{\mathcal{F}}

\newcommand{\cU}{\mathcal{U}}

\newcommand{\cV}{\mathcal{V}}

\newcommand{\cW}{\mathcal{W}}
\newcommand{\cP}{\mathcal{P}}

\newcommand{\cQ}{\mathcal{Q}}
\newcommand{\Pre}{\textnormal{Pre}}

\newcommand{\F}{\boldsymbol{F}}

\newcommand{\R}{\mathbb{R}}
\newcommand{\C}{\mathbb{C}}
\newcommand{\N}{\mathbb{N}}

\newcommand{\lra}{\longrightarrow}

\begin{document}

\title{On Gluing Data, Finite Ringed Spaces and schemes}

\maketitle

\centerline{R. Fioresi${}^{\star}$, A. Simonetti${}^\star$, F. Zanchetta${}^\star$}

\bigskip

\centerline{$^\star${\sl FaBiT, Universit\`{a} di
Bologna}}

\centerline{\sl Via Piero Gobetti 87, 40129 Bologna, Italy}

\begin{abstract} 
From descent theory to higher geometry, the idea of gluing has been embedded in many elegant and powerful techniques, proving instrumental for the solution of many problems. In this paper, we introduce a framework that allows to link important geometric objects, such as differentiable manifolds or schemes, to certain finite ringed spaces arising from sheaves on 2 dimensional semisimplicial sets, thus opening the door to their applications in fields such as discrete differential geometry.
\end{abstract}

\section{Introduction}
Discrete geometry has seen renewed interest over the last few years due to advances in machine learning and deep learning. Graph theory and methods from simplicial homotopy theory have been fundamental for the development of geometric deep learning (GDL)
\cite{br1, DLbook} and topological deep learning, \cite{TDLpos}, which extend deep learning to data with geometric or
topological structure. Indeed, algorithms like graph neural networks (GNNs) or simplicial neural networks (\cite{Spivak, bodnar2}), central in this area, are gaining traction and are becoming more widely used in applications. In addition, sheaves attached to these objects are receiving particular attention as well: sheaf neural networks (SNNs), now becoming an attractive field of research, rely on the theory of cellular sheaves \cite{HG19, curry}. \par
On the other side, algebraic and differential geometry very often describe their main
objects of interest by using the same combinatorics that discrete geometry is using. For example a differentiable manifold or scheme is often better understood as the \v{C}ech nerve of a covering of it: this simple yet crucial idea has been eventually formalized in descent theory, that played a key role in the theory of schemes, of stacks and in derived geometry. In this context, simplicial homotopy theory and category theory provide the background language for this notion to be used. These theories, very elegant and sophisticated, required the introduction of many technical tools that allowed the experts in the area to identify and prove many of the results that are now at the heart of geometry \cite{hartshorne, ja2001, lu2009}. However, these tools are not easily accessible to practitioners in neighbouring fields that do not have expert knowledge of these techniques. For example, simplicial
objects associated to manifolds or schemes are rarely perceived as ’data attached
to a discrete object’, therefore ready to used also in applied contexts, and their complexity is very often high. As a consequence, a
full transfer of the ideas and methods developed in fields such as differential geometry, higher geometry or algebraic geometry to discrete geometry has been limited and many discrete
geometric objects are often built in analogy with their non-discrete counterparts
and not from their non-discrete counterparts (or viceversa) in a more conceptual
way. \par Recent research in the field of algebraic geometry has started to address exactly this issue. In some recent work, Salas and their collaborators (see \cite{salas} and the references therein) study in more detail the notion of finite ringed spaces, i.e. ringed spaces having as underlying topological space a finite topological space. They find a suitable subcategory of finite ringed spaces that is equivalent, after localizing at certain weak equivalences, to the category of ordinary schemes; the relations between certain finite ringed spaces and differentiable manifolds are also considered, see \cite{salashomcoh} and the references therein. These works open the door to the understanding of highly structured objects, such as schemes, via simpler, in some sense discrete, objects. Indeed, because of the equivalence between finite topological spaces and partially ordered sets \cite{alex}, one readily notices that the datum of a finite ringed space effectively amounts to the datum of a poset and a sheaf of rings over it: at this point, the distance from the theory of cellular sheaves (that is the theory of sheaves, usually of vector spaces, on posets coming from regular cell complexes \cite{HG19}), recently introduced and used in discrete geometry and machine learning, becomes small. \par
The fact that from a sheaf of rings over a finite topological space we can build a
scheme does not come as a surprise as, under certain assumptions, one might think
to use the information encoded into discrete ringed spaces to glue together some
affine schemes, corresponding to the points of the ringed space. The information
contained in a finite topological space, or equivalently a partially ordered set, however, is not as structured or as readily informative as the one coming from a semisimplicial set or
a cell complex. These, however, arise naturally if we remind ourselves of the notion
of gluing datum (see \cite{GW}), that we can somehow think as a truncated Cech
nerve in dimension 2. Indeed, consider  a scheme (or, more generally, a ringed space) $X$, together with a cover of open affine subschemes $\lbrace U_i\subseteq X\rbrace_{i\in I}$. A ``gluing datum'' consisting
of the $U_i$'s, their intersections and identities coming from cocycle conditions then arises.
Gluing this datum amounts to ``glue'' the $U_i$'s along their intersections in such a way that the appropriate cocycle conditions, involving the transition maps, are satisfied and this process gives back a space isomorphic to $X$. Observe that we can build a 2 dimensional finite semisimplicial set out of $X$ taking the $U_i$'s as vertices, their intersections as edges and their triple intersections as 2-simplices.
We can then think of the process of gluing geometric objects as above as the process of gluing suitable data associated to this 2 dimensional semisimplicial set. Note that, to this end, no higher dimensional simplices are involved to reconstruct $X$. Let us focus now on the task of modeling schemes. Using the reasoning before as an heuristic argument, we identified a category, that we will denote as $\schtwo$, having as objects pairs consisting of a 2 dimensional finite semisimplicial set and a presheaf of rings on its face-incidence poset satisfying certain assumptions. A category having objects of this type is indeed desirable as in many parts of discrete geometry or in applications simple geometric objects such as graphs often provide enough geometric structure for the theory to be developed or for the experiments to be performed. Moreover
sheaf-based approaches to quantization in noncommutative geometry \cite{aflw, majidpaper, dimakis} can benefit by
such categorical characterization of algebraic geometric objects.\par
Our constructions resembles some others appearing in simplicial descent theory \cite{go12, ja2001}, but we were not able to find in the literature a category encoding the notion of gluing data as extra data attached to a low dimensional semisimplicial set explicitly spelled out and whose relation with finite ringed spaces is explored (usually simplicial objects in suitable categories are considered and finite ringed spaces are rarely used).
In addition, our use of the language of semisimplicial sets distinguishes our approach from the one of \cite{salas} and the one pursued in cellular sheaf theory: we link our objects to explicit 2 dimensional semisimplicial sets instead of posets that could possibly have higher complexity (or that can have non-geometric origin). 
It is worth noticing that Ladkani in \cite{SP06} studied the categories of sheaves over finite posets, while Liu in \cite{SQ13} studied stacks and torsors over quivers. Their work, however, is only partially related with ours. \par
We shall now describe the contents of the paper, highlighting our main results. \par
In Section \ref{graph-sec} we review standard material on semisimplicial sets, with a particular focus on graphs, and we establish our notation, see \cite{GJ1999, SP09, EW18} as references. 
\par
In Section \ref{fin-sec} we recast the theory of gluing data as the theory of certain functors of ringed spaces over the category of simplices of 2 dimensional semisimplicial sets. We study the properties of the resulting objects. Intuitively, we see gluing data as ``sheaves of ringed spaces over a graph". 
\par
In Section \ref{subsec:schemes} we recall some constructions of \cite{salas} and we introduce and study the category $\schtwo$. Then, we identify a class $\cW$ of ``weak equivalences" in this category
that can be promoted to what we call ``schematic right multiplicative system", a weaker notion of the ordinary right multiplicative system (see \ref{def:srms} and the discussion therein).
It then becomes possible to localize in a weak way, see Section \ref{subsec:schemes} for the details, the category $\schtwo$ at the class $\cW$ obtaining a category $\schtwo[\cW^{-1}]$. Notice that our localization procedure is slightly different from the ordinary one, since the natural candidate $\cW$ is not a right multiplicative system. We then obtain our main result.

\begin{theorem}[\ref{main-res}]
The category $\schtwo[\cW^{-1}]$ is equivalent to the category of quasi-compact and semi-separated schemes.
\end{theorem}

We also compare our constructions with the ones of Salas \cite{salas} and we point out to some possible future developments in the context of differential geometry while also making a few remarks useful to gain further insights on our work \cite{FSZ26}.

\medskip
{\bf Acknowledgments.}
The first author thanks Barbara Fantechi for helpful discussions. The first and last authors thank Francesco Vaccarino for helpful discussions and his interest in our work.
This research was supported by GNSAGA-Indam, INFN Gast Initiative,
PNRR MNESYS, PNRR National Center
for HPC, Big Data and Quantum Computing CUP J33C22001170001, PNNR SIMQuSEC
CUP J13C22000680006. 
This work was also supported by Horizon Europe EU projects MSCA-SE CaLIGOLA,
Project ID: 101086123, MSCA-DN CaLiForNIA, Project ID: 101119552,
COST Action CaLISTA CA21109.
\hyphenation{hos-pi-ta-li-ty}

\medskip
{\bf Notation.} We shall record some of the notations we employ throughout the paper here for the convenience of the reader.
\begin{itemize}
    \item We denote as ($\Delta_+$) $\Delta$
the category having as objects the partially
ordered sets $[n]=\lbrace 0\rightarrow\cdots\rightarrow n\rbrace$, $n\in\N$,
and arrows the (injective)
order preserving maps between them. For any $n\in\N$ e denote as
($\Delta_{n,+}$) $\Delta_{n}$
its full subcategory of objects having cardinality smaller or equal than $n$.
\item We will denote with $\Sets$ and $\Vectk$ the categories of sets and of finitely generated vector spaces over a base field $\mathbb{K}$ respectively. We will denote with $\Top$ the category of topological spaces and with $\RngSpcs$, $\LRngSpcs$ the categories of ringed spaces and locally ringed spaces respectively.
\item Given a set $X$ a preorder on $X$ is a binary relation that is reflexive and transitive. If this relation is also antisymmetric, then we have a partial order: in this case we will say that $X$ is a poset.
\item Given a category $\cC$ we will denote as $\cC^{op}$ its dual (or opposite) category. Finally, given two categories $\cA$ and $\cB$ we shall denote as $\Fun(\cA,\cB)$ or as $\cB^{\cA}$ the category of the covariant functors $\cA\rightarrow\cB$ and as $\Pre(\cA,\cB):=\Fun(\cA^{op},\cB)$ the category of contravariant functors $\cA\rightarrow\cB$. We refer to the objects of $\Pre(\mathcal{A},\mathcal{B})$ as
\textit{presheaves}.
If the category $\cB$ is the category $\Sets$ or it is clear from the context,
we shall denote $\Pre(\cC,\cB)$ simply as $\Pre(\cC)$.
\end{itemize}

\section{Semisimplicial sets, graphs and sheaves}\label{graph-sec}
In this section, we first recall what a (semi)simplicial set is and then we introduce some basic notation about graphs in this setting. Finally, we recall some basic notions about sheaves on a base and the Alexandrov topology associated with a preordered set.

\subsection{The categories of (semi)simplicial sets and digraphs}\label{ssets-sec}

\begin{definition}\label{ssets-def}
We define the category of \textit{(semi)simplicial} sets to be the category
$\sSets:=\Pre(\Delta,\Sets)$ ($\ssSets:=\Pre(\Delta_+,\Sets)$).
For a given (semi)simplicial set $X$, for every $n\in\N$, we shall denote the set $X([n])$ as $X_n$ and call its elements \textit{simplices}. Replacing the category of sets with an arbitrary category $\cC$ gives us the notion of a \textit{(semi)simplicial object in} $\cC$. We define \textit{(semi)cosimplicial sets} as $\Fun(\Delta,\Sets)$ ($\Fun(\Delta_+,\Sets)$) and $n$-dimensional (semi)simplicial sets and (semi)cosimplicial sets as $\Pre(\Delta_n,\Sets)$, ($\Pre(\Delta_{n,+},\Sets)$), $\Fun(\Delta_n,\Sets)$, ($\Fun(\Delta_{n,+},$ $\Sets)$) respectively. We say that a (co)semisimplicial set $X$ has dimension $n$ if $X_n\neq\emptyset$ and $X_m=\emptyset$ for $m>n$. 
\end{definition}

\begin{remark}
By definition, our $n$-dimensional simplicial sets are not simplicial sets, while in literature a simplicial set is usually said to have dimension $n$ if it is isomorphic to its $n$-skeleton \cite{GJ1999}. However, the inclusion functor $\Delta_n\subseteq\Delta$ induces a truncation function $tr_n:\sSets\rightarrow\sSets_n$ having as a fully faithful left adjoint $sk_n:\sSets_n\rightarrow\sSets$ the $n$-skeleton functor. For a given simplicial set $X$, $sk_n\circ tr_n (X)$ is the usual $n$-skeleton (see also \cite[\href{https://kerodon.net/tag/04ZY}{Tag 04ZY}]{kerodon}). This recovers the link between our definition and the most common definition.
\end{remark}

\begin{definition}
We define the category of \textit{directed graphs} to be $\Pre(\Delta_{1,+},\Sets)$ and
we denote it $\diGraphs$. Replacing the category of sets with an arbitrary category $\cC$ gives us the notion of \textit{directed graphs in} $\cC$.
\end{definition}

Given a digraph $G$, we will call the sets $G_0$ and $G_1$ (or $V_G$ and $E_G$) the set of \emph{vertices} and \emph{edges} respectively. We can represent a digraph using the following diagram, where the maps $h_G$ and $t_G$ are called head and tail.
\begin{displaymath}
\xymatrix{E_G\ar@<1ex>[r]^{h_G} 
\ar@<-1ex>[r]_{t_G}& V_G
}\end{displaymath}
Note that this definition of directed graph allows graphs with multiple edges with same head and tail (sometimes called \emph{directed multi-graphs}) and self-loops. We denote with $\diGraphs_{\leq1}$ the full subcategory of $\diGraphs$ whose objects have at most one edge connecting each pair of vertices. In addition, note that for a given $G \in \diGraphs$ both $V_G$ and $E_G$ can be infinite sets.

\begin{remark}\label{dim1simp}
Semisimplicial sets and simplicial sets are closely related. The latter category is the most frequently used in the context of simplicial homotopy theory because of the nice properties of the geometric realisation \cite{May1992, GJ1999}.
The former appeared more frequently in early modern algebraic geometry, for example in \cite{SGA4}. For the purpose of this work, semisimplicial objects provide a more natural context because of their easier combinatorics. The reader must be assured, however, that they are not an artificial category: there exists a geometric realisation for semisimplicial objects (the ``fat realisation'') and the geometric realisation of a simplicial set $X$ is always homotopy equivalent to the geometric realisation of the semisimplicial set obtained from $X$ precomposing with the inclusion $\Delta_+^{op}\subseteq \Delta^{op}$ (see \cite[Section 2]{EW18} and the discussion therein). Moreover in dimension lower or equal to one the categories
of simplicial sets and semisimplicial sets are equivalent, \cite[\href{https://kerodon.net/tag/001N}{Proposition 001N}]{kerodon}, thus directed graphs may be viewed as simplicial sets. 
\end{remark}

As semisimplicial sets are simply presheaves of sets, we have the following \cite{KS2006}.

\begin{proposition}
The category of semisimplicial sets has all limits and colimits and they are computed pointwise.
\end{proposition}

In particular, fibre products of semisimplicial sets exist. We define the category of \textsl{(semi)simplices} of $X$, $\Gamma(X)$, for a given (semi)simplicial set $X$, see \cite{hovey}. We recall briefly its construction.

For a given (semi)simplicial set $X$, by Yoneda's Lemma, the maps $\Delta^n\rightarrow X$ (resp.  $\Delta_+^n\rightarrow X$) are in bijection with the elements of $X([n])=X_n$,
where $\Delta^n=\Hom_\Delta(-,[n])$ (resp. $\Delta^n_+=\Hom_{\Delta_+}(-,[n])$). So we take such maps as the objects for $\Gamma(X)$ as $n$ varies. To define the morphisms of $\Gamma(X)$, we specify when we have an arrow between two objects in $\Gamma(X)$ that is,$a \in X_n$ and $b \in X_m$. We say we have such an arrow, whenever we can write a commutative diagram:
\begin{equation}\label{arr-sim-eq}
\xymatrix{\Delta^n \ar[rr] \ar[dr]_a & & \Delta^m\ar[dl]^b\\
 & X &}
\qquad \hbox{resp.} \quad
\xymatrix{\Delta^n_+ \ar[rr] \ar[dr]_a & & \Delta_+^m\ar[dl]^b\\
 & X &}
\end{equation}
where again we employ Yoneda's lemma.

\begin{definition}\label{catofsimp}
For a given (semi)simplicial set $X$, we define $\Gamma(X)$ the \textit{category of (semi)simplices of} $X$ as the category with objects the maps $\Delta^n\rightarrow X$ (resp.  $\Delta_+^n\rightarrow X$) and morphisms the diagrams in (\ref{arr-sim-eq}). Notice that a morphism $f:X\rightarrow Y$ of (semi)simplicial sets induces a functor
between their respective categories of simplices that we will denote as $\Gamma(f)$.
\end{definition}

\begin{observation}\label{digraph-poset}
For a semisimplicial set $X$, the objects and the arrows of $\Gamma(X)$ can be used to define a poset $P_X$ where the elements are the simplices of $X$ and $x\leq y$ if and only if $\Hom_{\Gamma(X)}(x,y)$ is non empty.
\end{observation}

\begin{remark}\label{cat-poset}
Notice that for a semisimplicial set $X$ there is a non-canonical isomorphism between the category $\Gamma(X)$ and the category associated with the poset $P_X$ having as objects the elements of $P_X$ and arrows $x \leq y$.
\end{remark}

Recall that an undirected graph $G$ is a pair $(V,E)$ of vertices and edges, where the edges are defined as a multiset of unordered pairs of vertices and morphisms are defined accordingly (see \cite{diestel, godsil} for the case $V$ and $E$ finite sets, though here we are more general).
Let $\Graphs$ denote their category. We call the category of simple undirected graphs, i.e. undirected graphs having at most one edge connecting each pair of vertices, as $\stGraphs$.

\begin{observation}\label{graph-poset} In Observation \ref{digraph-poset} we have given a way to associate a poset to a semisimplicial set, and therefore to a directed graph.
We can do the same with undirected graphs as follows.
Let $G=(V,E)\in\Graphs$ be an undirected graph. Define the poset $\mathcal{P}_G$ associated to it having as underlying set $V\cup E$ and where $x\leq y$ if and only if $x$ is a vertex of the edge $y$ or $x=y$.

\end{observation}

\subsection{Sheaves on preordered sets}
\label{sec-preorder}
A topological space is called \textit{finite} if it consists of a finite number of points. For a given finite space $X$ and a point $p\in X$, we define
$$U_p:=\textnormal{ smallest open subset of }X\textnormal{ containing }p$$
$$C_p:=\overline{p}=\textnormal{ smallest closed subset of }X\textnormal{ containing }p$$
Given a finite topological space $X$, we can define the structure of a (finite) preorder on $X$ by setting $p\leq q$ if and only if $U_p\supseteq U_q$ (equivalently, $p\in \overline{q}$). Conversely, given a (not necessarily finite) preorder $\leq$ on a set $P$, we can see $P$ as a topological space, whose topology is generated by the base consisting of the following open sets:
$$U_p:=\lbrace q\in P\: |\: q\geq p\rbrace\qquad p\in P$$
This topology is called the \textit{Alexandrov topology} \cite{alex} (see also \cite{salas}) and we denote as $\cB_P$ the base we used to generate it.

\begin{notation}\label{notation:alexandrovtop}
Whenever necessary, to avoid confusion, we shall use a different notation for a \textit{preorder} $P$ (i.e. a preordered set) and the topological space $A(P)$ associated with it, as above. Similarly, for a finite topological space $X$ we will write $P(X)$ for the preorder associated with it as above. We note explicitly that if $X$ is a finite topological space and $P(X)$ is the associated preorder, then $X=A(P(X))$. We denote by
$\mathrm{Open}(X)$ the category of the open sets of a topological space $X$. Recall that sheaves on $X$ are defined as presheaves on $\mathrm{Open}(X)$ such that the identity and the gluability axioms hold (see for example \cite{vakil} Definition 2.2.6).
\end{notation}

We have the following result, see \cite{salas} for a clear review of these statements.

\begin{theorem}\label{Thm:alextoppos}
The following statements hold:
\begin{itemize}
\item[1)] The above construction defines an equivalence of categories between finite topological spaces $\mathrm{FTop}$ and finite preordered sets $\mathrm{PreSets}$ given by the functors:
$$P: \mathrm{FTop}\lra \mathrm{PreSets}, \qquad A:\mathrm{PreSets}\lra \mathrm{FTop}$$
\item[2)] A finite topological space $X$ is $T_0$ (i.e. different points have different closures) if and only if the preorder relation $\leq$ induced by the topology is antisymmetric i.e. $X$ is a poset.
\end{itemize}
\end{theorem}

\begin{remark}\label{rmk:irred}
For a given finite topological space $X$, the irreducible open and
closed sets are $U_p$ and $C_p$, $p\in X$.
The same holds true if $X$ is a topological space of the form $A(P)$
for some (not necessarily finite) preorder $P$.
\end{remark}

\begin{observation}\label{obs:gross}
If $X$ is a semisimplicial set, we obtain the poset $P_X$ as in Obs. \ref{digraph-poset},
hence we can consider sheaves on $A(P_X)$. It is well known (see for example \cite{vakil} 2.5.1) that sheaves on a topological space are uniquely determined by a sheaf on a base for the underlying topological space. In our case, the base $\cB_{P_X}$ of $A(P_X)$ consists of irreducible open subsets of $A(P_X)$. As a consequence presheaves on the base $\cB_{P_X}$ are automatically sheaves on a the base $\cB_{P_X}$. In addition if we consider the full subcategory of $\mathrm{Open}(A(P_X))$ having as objects the elements of $\cB_{P_X}$, we see that it is isomorphic to $P_X$. Putting all together, we get the following proposition.
\end{observation}

\begin{proposition}\label{prop:gross}
Let be $X$ a semisimplicial set or a finite topological space. Consider its poset $P_X$. Then the category of sheaves on $A(P_X)$ is equivalent to the category $\Pre(P_X)$.
\end{proposition}
The previous proposition holds not only for the case of sheaves of sets but also for sheaves taking values in other categories, such as abelian categories, etc.

\begin{example}\label{graph-alex}
Let $G=(V_G,E_G)$ be in $\Graphs$. The topology of the space $A(\cP_G)$ is then generated by the following base of open sets:
\begin{itemize}
    \item $U_v=\lbrace e\in E_G\: |\: v\leq e\rbrace\cup\{v\}$, that is the open star of $v$, for each vertex $v\in V$,
\item $U_e=\lbrace e\rbrace$, i.e. the edge $e$, without its vertices, for each $e\in E$.
\end{itemize}
\end{example}
From now until the end of the paper we shall make the blanket assumption that all the semisimplicial sets we consider are finite unless stated otherwise.

\section{Finite ringed spaces, gluing data and semisimplicial sets}\label{fin-sec}
Gluing data as in \cite{GW}, \cite{SP} allow us to define objects like manifolds, schemes or morphisms between them by ``gluing'' some local data satisfying certain compatibility conditions. The aim of this section is to study this notion, via the semisimplicial set language of the previous section, in a way that can be useful for applications.

\subsection{Gluing data}\label{GDrngspcs}
We recall the definition of gluing datum of ringed spaces, (see 3.5 of \cite{GW}, or 6.33 of \cite{SP}).

A ringed space is a pair $(X,\cO_X)$ consisting of a topological space $X$ and of a sheaf of rings $\cO_X$ on it; a gluing datum recovers a ringed space through a collection of ringed spaces and compatibility conditions. 

\begin{definition}\label{gluing-datum}
A \emph{gluing datum of ringed spaces} consists of:
\begin{itemize}
\item[$\bullet$] a collection of ringed spaces $(U_i,\cO_{U_i})_{i\in I}$;
\item[$\bullet$] a collection of open subspaces
$( U_{ij}\subseteq U_i)_{i,j\in I}$, $U_{ii}=U_i$, for all $i\in I$.

\item[$\bullet$] a collection of isomorphisms of ringed spaces $( \phi_{ji}:U_{ij}\rightarrow U_{ji})_{i,j\in I}$ such that: 
$$\phi_{ki}=\phi_{kj}\circ \phi_{ji}\quad \mathrm{on}\quad U_{ij}\cap U_{ik}, \qquad \hbox{(cocycle condition)}$$
\end{itemize}
We will denote such datum as $((U_i)_{i\in I},( U_{ij})_{i,j\in I},(\phi_{ij})_{i,j\in I})$.
\end{definition}

We may replace the sheaf of rings with a sheaf of abelian groups, sets, vector spaces etc. and obtain similar notions of gluing datum, for which the constructions we give below will hold.

As customary, we may sometimes denote both the topological space and the
ringed space with the same letter, whenever there is no danger of confusion. Replacing ringed spaces and morphisms between them with topological spaces and open embeddings (resp. schemes and open embeddings and so on) gives rise to the notion of gluing datum of topological spaces (resp. schemes, etc.). To fix the ideas, we will examine the case of gluing data in the category of ringed spaces $\RngSpcs$, but many of the definitions, lemmas and propositions we will prove, hold also replacing $\RngSpcs$ with other categories (like topological spaces $\Top$, schemes $\Sch$, etc.). From now on if we have two open embeddings $f:U\hookrightarrow X$, $g:V\hookrightarrow X$ we shall use the intuitive notation $U\cap V$ instead of the more precise one $U\times_XV$.

\begin{definition}\label{gluing-rs}
A ringed spaced $X$ is obtained by \textit{gluing} a gluing datum
$((U_i)_{i\in I},$ $( U_{ij})_{i,j\in I},$ $(\phi_{ij})_{i,j\in I})$, if there exists an open cover of $X$ of ringed subspaces $V_i$, such that $V_i \cong U_i$ and such isomorphisms restrict suitably to give $V_i \cap V_j \cong U_{ij}$ and the cocycle conditions.
\end{definition}

Given a gluing datum, we can always find a ringed space obtained by gluing it (see \cite{GW} or \cite{SP} for example). The gluing construction essentially amounts to taking a colimit i.e. direct limit of a certain diagram in the category of ringed spaces: we shall now explore this statement, so that we fit it into our discussion. This is a generally known fact, however we prefer to recast it so that it becomes useful in the study of finite ringed spaces (as in \cite{salas}).

We begin with the notion of {\it gluing cube} and we consider, for the moment, just the problem of gluing three ringed spaces.

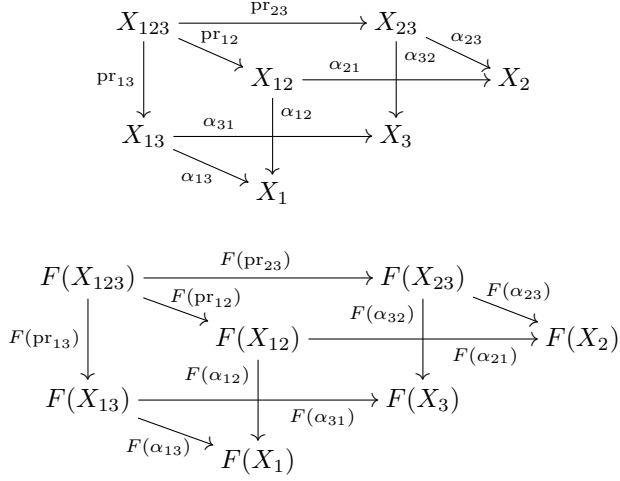
\begin{figure}\label{cubecat-fig}
\begin{displaymath}
\begin{array}{c}
\xymatrix@R0.25cm{
X_{123} \ar[dd]_{\pr_{13}} \ar[rr]^{\pr_{23}} \ar[dr]^{\pr_{12}} & & X_{23} \ar[dd]^(.3){\alpha_{32}} \ar[dr]^{\alpha_{23}} & \\
& X_{12} \ar[dd]^(.3){\alpha_{12}} \ar[rr]^(.3){\alpha_{21}} & & X_2 \\
X_{13}  \ar[dr]_{\alpha_{13}} \ar[rr]^(.3){\alpha_{31}} & &X_3 &\\
& X_1 & &  } 
\\ \\
\xymatrix@R0.25cm{
F(X_{123}) \ar[dd]_{F(\pr_{13})} \ar[rr]^{F(\pr_{23})} \ar[dr]^(.6){F(\pr_{12})} & &
F(X_{23}) \ar[dd]_(.3){F(\alpha_{32})}
\ar[dr]^{F(\alpha_{23})} & \\
& F(X_{12}) \ar[dd]_(.3){F(\alpha_{12})} \ar[rr]_(.7){F(\alpha_{21})} & & F(X_2) \\
F(X_{13})  \ar[dr]_{F(\alpha_{13})} \ar[rr]_(.7){F(\alpha_{31})} & &F(X_3) &\\
& F(X_1) & &  }
\end{array}
\end{displaymath}
\caption{The category $\cube$ and a gluing cube}
\end{figure}

\begin{definition}\label{glcube}
We define the category $\cube$, as the index category consisting of $7$ objects $X_1$, $X_2$, $X_3$, $X_{12}$, $X_{13}$, $X_{23}$, $X_{123}$
and arrows as in Fig. \ref{cubecat-fig}, identity arrows not depicted. 
A functor $F:\cube\rightarrow\RngSpcs$ is called \textit{gluing cube}, if all the morphisms in $\cube$ are sent to open embeddings and the 3 commutative squares of diagram in Fig. \ref{cubecat-fig} are sent to cartesian squares.

\end{definition}
Given a gluing cube $F$, we can glue the ringed spaces $F(X_1)$, $F(X_2)$ and $F(X_3)$ along the ``intersections" (fibered products) $F(X_{ij})$ in such a way that a cocycle condition is satisfied, obtaining a ringed space $Y$, that visually occupies the ``missing vertex'' in the gluing cube
diagram \ref{cubecat-fig}. This process amounts to take the colimit of the gluing cube $F$. The gluing construction yields natural open embeddings $F(X_i) \hookrightarrow Y$ (see \cite{SP} 6.33).
The next observation clarifies these statements.

\begin{observation}\label{gluingcubelemma}
A gluing cube $F:\cube\rightarrow\RngSpcs$ defines a gluing datum $\mathrm{GD}(F)$ \break $=$ $((U_i),( U_{ij})$, $(\phi_{ij}))_{i,j\in\lbrace 1,2,3\rbrace}$ and viceversa.
To construct a gluing datum from $F$, we set $U_i=U_{ii}:=F(X_i)$ for $i=1,2,3$. We then define the $U_{ij}$ as follows. We know that the map $F(\alpha_{ij}):F(X_{ij})\rightarrow F(X_i)$ is an open embedding, so it factors through an open embedding $e_{ij}:(U_{ij}, \cO_{F(X_i)|U_{ij}})\rightarrow F(X_i)$ where $U_{ij}$ is an open subset of $F(X_i)$. This defines the desired $U_{ij}$. In addition, we have isomorphisms $\psi_{ij}:F(X_{ij})\xrightarrow{\cong}U_{ij}$ such that $e_{ij}\circ\psi_{ij}=F(\alpha_{ij})$.
Note here that if $i>j$, strictly speaking, we do not have a $X_{ij}$:
in this case by $F(X_{ij})$ we mean $F(X_{ji})$.
We set $\phi_{ji}:= \psi_{ji}\circ\psi_{ij}^-1: U_{ij}\xrightarrow{\cong}U_{ji}$. One can check that this defines a gluing datum, that we denote
$\mathrm{GD}(F):=((U_i),( U_{ij}), (\phi_{ij}))_{i,j\in\lbrace 1,2,3\rbrace}$.

Conversely, if we start from a gluing datum involving only three spaces, $W=((U_i),( U_{ij}),(\phi_{ij}))_{i,j\in\lbrace 1,2,3\rbrace}$, we
can define a gluing cube $F_W$ as follows. We set $F_W(X_i)=U_i$,  $F_W(X_{ij}):=U_{ij}$
and $F_W(\alpha_{ij})$ to be the inclusion $U_{ij}\subseteq U_i$ for $i,j=1,2,3$, $i< j$.
If $i>j$ we define $F_W(\alpha_{ij})$ to be the composition of $\phi_{ji}$ and the canonical inclusion
$U_{ij}\subseteq U_i$. Finally, we define $F_W(X_{123})$ and $F_W(\pr_{ij})$ by
taking the pullbacks of the maps $F_W(\alpha_{ij})$ in $\RngSpcs$ and using the cocycle conditions.

If the gluing datum $W$ comes from a gluing cube $F$, one can then easily check that this procedure gives back the gluing cube $F$, that is $F_{\mathrm{GD}(F)}\cong F$. Conversely, starting from a gluing datum $W$ {as above}, if we construct a gluing cube and then a gluing datum from it, we obtain $W$ back.
\end{observation}

\subsection{Gluing data for 2-dimensional semisimplicial sets}
We now consider gluing data involving more than 3 spaces.
Recall that $\Delta^2$ is isomorphic to the following semisimplicial set, that we shall call $T$:
\begin{displaymath}
    \xymatrix{
T_2= \lbrace A \rbrace \ar@<3ex>[r]^{\delta_0\qquad}
\ar[r]^{\delta_1\qquad} \ar@<-3ex> [r]^{\delta_2\qquad} &
T_1=\lbrace e_{12}, e_{23}, e_{31}\rbrace \ar@<1ex>[r]^{\delta_0'} \ar@<-1ex> [r]_{\delta_1'} &
T_0=\lbrace v_1, v_2, v_3 \rbrace
}\end{displaymath}
As usual $T_i$ denotes $T([i])$ and $\delta_i$, $\delta_i'$
are the usual face maps satisfying the face identities \cite{GJ1999}. 
The semisimplicial set $T$ can be represented as the following diagram:
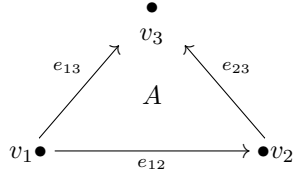
\begin{displaymath}\label{cubecat}
\xymatrix{
  & \underset{\begin{array}{c} v_3 \\ \\ A\end{array} }{\bullet} & \\
 v_1\bullet \ar[rr]_{e_{12}} \ar[ur]^{e_{13}} & & \bullet v_2 \ar[ul]_{e_{23}}
  }
\end{displaymath}

We have that the opposite category of the category of simplices $\Gamma(T)$, whose objects are the elements in $T_0$, $T_1$, $T_2$ (see Definition \ref{catofsimp}) is isomorphic to $\cube$ and so is the opposite of the category arising by the poset $P_T$ (see Obs. \ref{digraph-poset}). As a consequence, a gluing cube can be seen as a functor $(\Gamma(T))^{op}\cong (P_T)^{op}\rightarrow \RngSpcs $.

\medskip\begin{definition}
    $X\in\ssSets$ is called \textit{regular} if for every $n$-simplex $a\in X_n$, the smallest sub-semisimplicial set containing $a$ is isomorphic to $\Delta^n$ \cite{piccinini}.
\end{definition}

Let $G$ be {a finite graph in} $\diGraphss$ or $\stGraphs$ together with a total ordering of its vertices. Recall that the notion of $k$-cliques of $G$ is defined in both these categories: a $k$-clique is a subset of $k$ vertices such that each two distinct vertices are \textit{adjacent},
i.e. linked by an edge if $G$ is undirected or linked by an edge in each direction if $G$ is directed.

\begin{definition}\label{def:cliquesimp}
Let $G$ be either an object of $\diGraphss$ or $\stGraphs$. Given a total ordering on its vertices, we define the semisimplicial set $G_\bullet$ as the regular semisimplicial set of dimension 2 consisting of $1$, $2$ and $3$-cliques. Explicitly:
\begin{itemize}
\item $G_0$, $G_1$, $G_2$ are the sets of $1$-cliques (vertices), $2$-cliques
and $3$-cliques respectively.
\item Face maps are defined by ordering the complex of $1$, $2$ and $3$-cliques using the given total ordering of the vertices of $G$.
\end{itemize}
\end{definition}

Different choices of the total ordering on the vertices give rise to isomorphic $G_\bullet$. This follows from the fact that all total orderings on finite sets are isomorphic.

\begin{observation} \label{obs-grosscover}
The simplices of a semisimplicial set $G_\bullet$  are in bijection with sub-semisimplicial sets isomorphic to
$T:= \xymatrix{
  T_2\ar@<1ex>[r]\ar[r]\ar@<-1ex>[r] & T_1
 } \rightrightarrows T_0$,
$E:= \bullet\rightrightarrows \bullet$ and $P:=\bullet$ and are
represented respectively by the diagrams:
$$
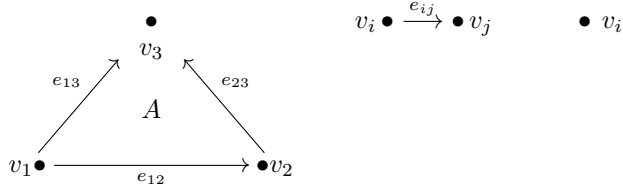
\xymatrix{
& \underset{\begin{array}{c} v_3 \\ \\ A\end{array} }{\bullet} & \\
v_1\bullet \ar[rr]_{e_{12}} \ar[ur]^{e_{13}} & & \bullet v_2
\ar[ul]_{e_{23}}
} \qquad
v_i\, \bullet \stackrel{e_{ij}}{\longrightarrow} \bullet \, v_j \qquad  \quad \bullet \, v_i
$$
The opposite categories of their categories of simplices $\Gamma(T)$, $\Gamma(E)$ and $\Gamma(P)$
are isomorphic respectively to the categories $\cube$, $\Lambda$ and $\ast$, defined by the diagrams below:
$$
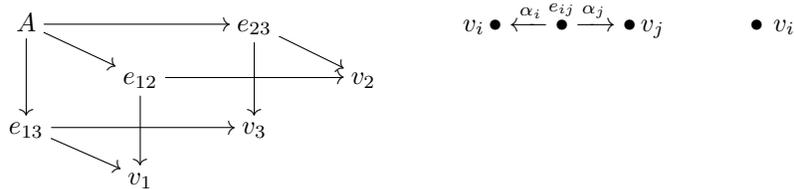
\xymatrix@R0.25cm{
A \ar[dd]_{} \ar[rr]^{} \ar[dr]^{} & & e_{23} \ar[dd]_{} \ar[dr]^{} & \\
& e_{12} \ar[dd]_{} \ar[rr]^{} & & v_2 \\
e_{13}  \ar[dr]_{} \ar[rr]^{} & &v_3 &\\
& v_1 & &  }\qquad
\quad v_i\, \bullet\xleftarrow{\alpha_i}\stackrel{e_{ij}}{\bullet}
\xrightarrow{\alpha_j} \bullet \, v_j \qquad \quad \bullet \, v_i
$$
As usual, the identity morphisms are not depicted in these diagrams.
\end{observation}

\begin{definition}
    We define a \emph{gluing wedge} and a \emph{gluing point} to be functors $\Lambda\rightarrow\RngSpcs$ and $\ast\rightarrow\RngSpcs$ respectively, such that all the morphisms in $\Lambda$ and $\ast$ are sent to open embeddings.
\end{definition}

\begin{remark}\label{gluingwedgespointsrmk}
Gluing wedges and gluing points give rise to gluing data $((U_i),( U_{ij}),$
$(\phi_{ij}))_{i,j\in\lbrace 1,2\rbrace}$ and $((U_1),$  $( U_{11}),$ $(\phi_{11}))$ respectively.
\end{remark}

We extend this remark to obtain a result {in Prop \ref{prop:gluingdatsimp}, key} for our subsequent treatment, which generalizes Obs. \ref{gluingcubelemma}, leaving all the details to the reader.

\begin{definition}\label{def:gluingfunctor}  
Given a graph $G$ we say that a functor $F:\Gamma(G_\bullet)^{op}\rightarrow\RngSpcs$ is a
\textit{gluing functor for $G$} if:
\begin{itemize}
\item[1.] For all $T,E,P\subseteq G_\bullet$, $F_{|\Gamma(T)^{op}}$, $F_{|\Gamma(E)^{op}}$, $F_{|\Gamma(P)^{op}}$
are respectively a gluing cube, wedge or point, where
$$\displaylines{
F_{|\Gamma(T)^{op}}:\Gamma(T)^{op}\rightarrow\RngSpcs,
\qquad F_{|\Gamma(E)^{op}}:\Gamma(E)^{op}\rightarrow\RngSpcs, \cr
F_{|\Gamma(P)^{op}}:\Gamma(P)^{op}\rightarrow\RngSpcs
}$$

\item[2.] For all pairs of gluing wedges $F(a)\leftarrow F(b)\rightarrow F(c)\leftarrow F(d)\rightarrow F(e)$, if $a$ and $e$ are not linked by an edge, the pullback of ringed spaces $F(b)\times_{F(c)}F(d)$ is the empty ringed space.
\end{itemize}
\end{definition}

Let $G$ be a finite
graph belonging to either the category $\diGraphss$ or $\stGraphs$,
together with a total ordering of its vertices and let $\Gamma(G_\bullet)$
be the category of simplices of $G_\bullet$. 

\begin{proposition} \label{prop:gluingdatsimp}

{Let $F$ be a gluing functor for a finite graph $G$ as above.}
$F$ defines a gluing datum $\mathrm{GD}(F)=((U_i),( U_{ij}),(\phi_{ij}))_{i,j\in\lbrace 1,...,n\rbrace}$
and $\colim F$ is isomorphic to the ringed space obtained by gluing $\mathrm{GD}(F)$.

Conversely, given a gluing datum $W=((U_i),( U_{ij}),(\phi_{ij}))_{i,j\in\lbrace 1,...,n\rbrace}$,
we can define a 2-dimensional semisimplicial set $G_\bullet$ whose 1-skeleton
is isomorphic to a graph $G\in\diGraphs$ having at most one edge joining each pair of vertices and a
functor $F_W:\Gamma(G_\bullet)^{op}\rightarrow \RngSpcs$,
such that $\colim F_W$ is isomorphic to the ringed space obtained by gluing $W$.
\end{proposition}

\begin{remark}\label{rmk:disjoint}
    Consider a gluing functor $F$ for a finite graph $G$. If two vertices $u$, $v$ of $G$ are not linked by any edge, then the ringed spaces $F(a)$ and $F(b)$ are disjoint open sub ringed spaces of $\colim F$. 
\end{remark}

\begin{observation}
    In the previous proposition, the choice of the total ordering on the vertices of $G$ is 
    immaterial: two distinct total orderings on $G$ give rise to an automorphism of $G_\bullet$ inducing on the colimits of the two resulting gluing functors an isomorphism.
\end{observation}

\subsection{Gluing Ringed Spaces}\label{sect:Gluingringedspaces}
Given a finite
graph $G$, we now show how the information encoded in gluing functor 
$F:\Gamma(G_\bullet)^{op}\rightarrow \RngSpcs$ for $G$
effectively recovers a ringed space.

\medskip

We have a fully faithful inclusion functor $\Top\subseteq\RngSpcs$ 
and forgetful functors $\LRngSpcs\rightarrow\RngSpcs\rightarrow\Top$. As these forgetful functors have right adjoints, they preserve colimits. In particular, whenever we glue a gluing datum of (locally) ringed spaces, the underlying topological space is obtained as if we were gluing a gluing datum of topological spaces. The same also applies if we replace the sheaves of rings with sheaves of abelian groups, etc.

\begin{definition}\label{def:grfin}
Consider a ringed space $(S, \cO_S)$ and a finite covering by open embeddings $\cU=\lbrace U_i\hookrightarrow S\rbrace_{i=1,...,n}$ of it.
We identify each $U_i$ with the subspace of $S$ isomorphic to it.
This covering gives rise to a gluing datum $(U_i, U_{ij}:=U_i \cap U_j, \phi_{ij})$ where the $\phi_{ij}$ are isomorphisms.
We define the graph $G(\cU)\in \Graphs_{\leq  1}$ by taking as vertices the ordered set of $U_i$'s and placing an edge between two vertices $v_i$ and $v_j$ $i\neq j$ if and only if $U_{ij} \neq \emptyset$.
\end{definition}

Notice that, since vertices of $G(\cU)$ are naturally ordered we can immediately construct the
semisimplicial set $G(\cU)_\bullet$ according to Def. \ref{def:cliquesimp}. It has one 0-simplex for each $U_i$, a 1-simplex for each non empty $U_i\cap U_j$ ($i<j$) and one 2-simplex for each triple intersection $U_i\cap U_j\cap U_k$ corresponding to a $3$-clique of $G(\cU)$ (these triple intersections can be empty).
Moreover, the 2-dimensional semisimplicial set $G(\cU)_\bullet$ is regular.

\begin{remark}\label{rmk:equalsubschemes}
The identification of each element of a covering $\cU$ by open embeddings of a given ringed space $S$ with the subspace of $S$ isomorphic to it is merely a way to avoid choosing a cleavage as in \cite{Vis}. Please note that this choice does not ``remove repetitions" in the following sense. If we have two schemes in $\cU$ that embed to the same open subspace of $S$ we still have in $G(\cU)_0$ two distinct vertices corresponding to these ringed spaces.
\end{remark}

Given $S$, $\cU$ as above we can define a finite ringed space $(S_\cU, \cO_{S_\cU})$ as in \cite{salas}.

\begin{definition}\label{def:constrFinSp}
Let $(S, \cO_S)$ be a ringed space and $\cU$ a finite covering by open embeddings. 
For each point $s\in S$, we define $U^s:=\cap_{s\in U_i}U_i$. Consider the
equivalence relation:
$$
s\sim s' \quad \hbox{if and only if} \quad U^s=U^{s'}
$$
We define the finite ringed space $(S_\cU, \cO_{S_\cU})$
associated to $S$ and the cover $\cU$ as $S_\cU:=S/\sim$ and $\cO_{S_\cU}:=\pi_\ast\cO_S$,
where $\pi:S\rightarrow S_\cU$ is the continuous projection morphism.
Notice that $\pi$ {can be promoted to} a morphism of ringed spaces.
\end{definition}

\begin{observation}\label{rmk:salasXdef}
Let $\pi:S\rightarrow S_\cU$ and $\cU$
{as in Def. \ref{def:constrFinSp}.}
One can check that the image of $U_i\cap U_j=:U_{ij}$ and $U_i\cap U_j\cap U_k=:U_{ijk}$ under $\pi$ is an open subset for every $i,j,k$. By looking at $U_i$, $U_{ij}$ and $U_{ijk}$ as ringed spaces, we define $V_i:=(\pi(U_i),(\pi_{|U_i})_\ast\cO_{S|U_i})$, $V_{ij}:=(\pi(U_{ij}),(\pi_{|U_{ij}})_\ast\cO_{S|U_{ij}})$ and $V_{ijk}:=(\pi(U_{ijk}),(\pi_{|U_{ijk}})_\ast\cO_{S|U_{ijk}})$. These are open sub-ringed spaces of $S_\cU$: the $V_i$ form a cover of $S_\cU$ by open embeddings and $V_{ij}$, $V_{ijk}$ are the double and the triple intersections of them so that they naturally define a gluing datum that returns $S_\cU$ after gluing.
\end{observation}
We will denote as $\RSfin$ the category of ringed spaces having as underlying topological space a finite topological space. For a given ringed space $S$ and a given open covering $\cU$ of it as above, using Proposition \ref{prop:gluingdatsimp} and Observation \ref{rmk:salasXdef} we can define:
\begin{itemize}
\item $F_{\mathrm{can}}^{\cU}: \Gamma(G(\cU)_\bullet)^{op}\rightarrow \RngSpcs$, sending the objects of $\Gamma(G(\cU))^{op}$ to $U_i$, $U_{ij}$, $U_{ijk}$.
\item $F_{\mathrm{fin}}^{\cU}:\Gamma(G(\cU)_\bullet)^{op}\rightarrow \RSfin$, sending the objects of $\Gamma(G(\cU))^{op}$ to $V_i$, $V_{ij}$, $V_{ijk}$.
\end{itemize}

\begin{lemma}\label{lemma:Fcan}
Let the notation be as above. Then
$$
\colim F_{\mathrm{fin}}^{\cU}\cong S_\cU \qquad
\colim F_{\mathrm{can}}^{\cU}\cong S.
$$
In addition we have a natural transformation of functors $\Tilde{\pi}:F_{\mathrm{can}}^{\cU}\rightarrow F_{\mathrm{fin}}^{\cU}$ induced by $\pi$ that, under the isomorphisms $\colim F_{\mathrm{fin}}^{\cU}\cong S_\cU$ and $\colim F_{\mathrm{can}}^{\cU}\cong S$ reconstructs the morphism $\pi:S\rightarrow S_\cU$ on the colimits.
\end{lemma}
\begin{proof}
We leave to the reader to check the compatibility conditions.
\end{proof}

We end this section by constructing a ringed space similar, but different
from $(S_\cU, \cO_{S_\cU})$ in Def. \ref{def:constrFinSp}, that will be important later.

\begin{observation}\label{rmk:gbullet}
If  $(S, \cO_S)$ is a ringed space and $\cU=\{U_i\}$ is a
finite covering of it by open embeddings, $\Gamma(G(\cU)_\bullet)$ is a category isomorphic to the category associated with the poset $P_{G(\cU)_\bullet}$ (see Observation \ref{digraph-poset},
Remark \ref{cat-poset}). Unraveling the definitions, we see that this poset has one element $p_i$ for each $U_i$, an element $p_{ij}$ for each $U_{ij}$ ($i< j$, $U_i\cap U_j\neq \emptyset$) and an element $p_{ijk}$ for each $U_{ijk}$ corresponding to a $3$-clique of $G(\cU)$.
We can consider the topological space $A(P_{G(\cU)_\bullet})$ associated to the poset $P_{G(\cU)_\bullet}$ using the Alexandrov topology (see \ref{notation:alexandrovtop}).
One can check, by the very definition of the Alexandrov topology,
that $P_{G(\cU)_\bullet}$, viewed as a category, is isomorphic to the dual category of irreducible open subsets of $A(P_{G(\cU)_\bullet})$.
This isomorphism associates to each $p_i$, $p_{ij}$, $p_{ijk}$ the irreducible open subsets $U_{p_i}$, $U_{p_{ij}}$ and $U_{p_{ijk}}$ of  $A(P_{G(\cU)_\bullet})$ (see Section \ref{sec-preorder} for the notation).
Hence, a functor from $P_{G(\cU)_\bullet}$ to $\Rings$ will yield a sheaf on the topological space $A(P_{G(\cU)_\bullet})$ (see Proposition \ref{prop:gross}). Hence, the functor
$$
F_{\cU}^2: P_{G(\cU)_\bullet}\cong\Gamma(G(\cU)_\bullet)\rightarrow \Rings, \quad
p_i\mapsto \cO_S(U_i),\, p_{ij}\mapsto \cO_S(U_{ij}),\, p_{ijk}\mapsto\cO_S(U_{ijk})
$$
is a sheaf of rings on the base {$\{U_{p_i}, U_{p_{ij}}, U_{p_{ijk}}\}$}
of $A(P_{G(\cU)_\bullet})$. In other words, $F_\cU^2$ is defined by composing the functor $F^\cU_{\mathrm{can}}$ with the global section functor.
\end{observation}

\begin{definition}\label{Fund::preschspcover}
Let $S$ be a ringed space and let $\cU$ be a finite covering of it by open embeddings. We define the ringed space $(S_\cU^2,\cO_{S_\cU^2})$:
\begin{itemize}
\item $S_\cU^2$ is the topological space $A(P_{G(\cU)_\bullet})$.
\item $\cO_{S_\cU^2}$ is the sheaf of rings on $S_\cU^2$ obtained from the sheaf $F_{\cU}^2$ of Observation \ref{rmk:gbullet} on a base of $A(P_{G(\cU)_\bullet})$.
\end{itemize}
We will say also in this case that $(S_\cU^2,\cO_{S_\cU^2})$ is the ringed
space associated to $S$ and the cover $\cU$.
\end{definition}

\begin{remark}
Given a ringed space $S$ and a finite covering by open embeddings $\cU$, in general the topological spaces $S_\cU^2$ and $S_\cU$ are {not} homeomorphic and consequently the two resulting ringed spaces according to Definitions \ref{def:constrFinSp} and \ref{Fund::preschspcover} are not isomorphic. Indeed, consider a covering $\cU$ consisting of $S$ itself and of a given {non empty} open subspace $U\subset S$. Then $S_\cU$ is a finite topological space consisting of two points while $S_\cU^2$ is a finite topological space consisting of 3 points.
\end{remark}

\begin{remark}
Given a covering $\cU$ of $S$ as above one can define the sub-semisimplicial set $i:\widetilde{G}(\cU)_\bullet\hookrightarrow G(\cU)_\bullet$ obtained by $G(\cU)_\bullet$ by removing in $G(\cU)_2$ the elements $p_{ijk}$ associated with triple intersections $U_{ijk}$ that are empty. From $i$ we get a map $\Gamma(i):\Gamma(\widetilde{G}(\cU)_\bullet)\rightarrow \Gamma(G(\cU)_\bullet)$ and
therefore, as $P_{G(\cU)_\bullet}\cong\Gamma(G(\cU)_\bullet)$ and $P_{\widetilde{G}(\cU)_\bullet}\cong\Gamma(\widetilde{G}(\cU)_\bullet)$, a functor:
$$
\widetilde{F}_{\cU}^2:=F_{\cU}^2\circ\Gamma(i), \quad \widetilde{F}_{\cU}^2:
P_{G(\cU)_\bullet}\cong\Gamma(\widetilde{G}(\cU)_\bullet)\rightarrow \Rings
$$
From $i$ we also get a topological space $\widetilde{S}_\cU^2$, a continuous inclusion $f:\widetilde{S}_\cU^2\hookrightarrow S_\cU^2$, a sheaf $\cO_{\widetilde{S}_\cU^2}$ on $\widetilde{S}_\cU^2$ obtained from $\widetilde{F}_{\cU}^2$ and a functor $\widetilde{F}_{\mathrm{can}}^\cU=F_{\mathrm{can}}^\cU\circ\Gamma(i)^{op}$. As for every sheaf $\cF$ on a topological space we have that $F(\emptyset)$ is equal to the zero ring, one can check that by the very definition of $\cO_{\widetilde{S}_\cU^2}$ we have an isomorphism $\varphi:\cO_{S_\cU^2}\xrightarrow{\cong}f_\ast \cO_{\widetilde{S}_\cU^2}$ giving rise to a morphism of ringed spaces $(f,\varphi):(\widetilde{S}_\cU^2,\cO_{\widetilde{S}_\cU^2})\rightarrow(S_\cU^2,\cO_{S_\cU^2})$ such that $\varphi$ is a sheaf isomorphism. This morphism of ringed space is not an isomorphism but its latter property, together with the fact that the colimits of $\widetilde{F}_{\mathrm{can}}^\cU$ and $F_{\mathrm{can}}^\cU\circ\Gamma(i)^{op}$ can be checked to be both isomorphic to $S$ might prompt us to think that $(\widetilde{S}_\cU^2,\cO_{\widetilde{S}_\cU^2})$ and $(S_\cU^2,\cO_{S_\cU^2})$ are weakly equivalent in some sense to be defined. This heuristic argument motivates part of the reasoning found in Section \ref{sect:schtwo}.
\end{remark}

\section{Finite ringed spaces, schemes and semisimplicial sets}\label{subsec:schemes}
In this section we specialize some of the previous constructions to the case of schemes and manifolds. We also introduce a convenient category to study gluing data of schemes leveraging the language of semisimplicial set that we will cal $\schtwo$.

\subsection{Finite ringed spaces and schemes} \label{subsec:schemesrs}
In this section we study more in detail the gluing constructions introduced in the previous section for the case of schemes. For the main definitions as scheme, Spec, etc. we refer the reader to \cite{hartshorne} Ch. 2. Our next definition is directly inspired by the theory of {\sl schematic spaces} (see \cite{salas} and refs therein and also Sec. \ref{sec-mflds} for the differentiable category).

Consider a finite ringed space $(X,\cO_X)$, i.e. a ringed space having as underlying topological space a finite topological space. There is a natural preorder on $X$ as explained in Section \ref{sec-preorder}, Theorem \ref{Thm:alextoppos}.

\begin{definition}\label{paraschematic}
We say that a finite ringed space $(X,\cO_X)$ is {\it paraschematic} if
for all $p\geq q$, $p,q\in X$, (recall that $p\geq q$ if and only if $U_q \subset U_p$), the scheme morphism
$\Spec(\cO_X(U_q))\rightarrow\Spec(\cO_X(U_p))$ is an open embedding.
We denote the category of paraschematic spaces as  $\mathrm{PSch}$.
\end{definition}

Recall that a scheme is called semi-separated if its diagonal morphism is an affine morphism or, equivalently, if the intersection of any two affine subschemes of it is an affine scheme. Recall also that a scheme is called quasi-compact if admits a finite cover of affine open subschemes.
Notable categories of schemes enjoy these properties, for example all quasi-projective schemes (over a regular ring, say) or more generally all divisorial schemes have this property. We denote the category of quasi-compact and semi-separated schemes with $\mathrm{Sch}$. Unless otherwise specified, we assume schemes to be in this category.

We now make an observation regarding the finite ringed space $(S_\cU, \cO_{S_\cU})$ as introduced in Def. \ref{def:constrFinSp}, for the case in which $S$ is a scheme. 

\begin{observation}\label{obs:su}
Let $S$ be a scheme together with a finite covering by affine open subschemes $\cU$. As before, we can view the preorder $P(S_\cU)$ as a category and it is isomorphic to the dual category of irreducible open subsets in $S_\cU$. Then, the functor:
$$P(S_\cU)\lra \Rings, \qquad p \lra \cO_{S_\cU}(U_p)$$
is a (pre)sheaf {on a base of $S_\cU$}, where
$U_p$ denotes the smallest open set in $S_\cU$ containing $p$ (see Section \ref{sec-preorder}). Moreover {its extension to the whole $S_\cU$} is $\cO_{S_\cU}$.
\end{observation}

We can repeat the construction in Obs. \ref{obs:su} replacing $S_\cU$ by a paraschematic space $X$, thus obtaining a presheaf:
$$P(X) \lra \Rings,\qquad   p \lra \cO_{X}(U_p)$$
extending to the sheaf $\cO_X$ on $X$. This leads us to the following definition \cite{salas}.

\begin{definition}\label{def:SpecS}
Let $(X,\cO_X)$ be a paraschematic space.
We define the functor:
$$
F_{X}:(P(X))^{op} \rightarrow
{\RngSpcs}, \qquad p \mapsto \mathrm{Spec}\, \cO_{X}(U_p)
$$
This allows us to define the functor:
$$
\mathrm{Spec}^S:\mathrm{PSch} \lra
{\RngSpcs}, \qquad X \mapsto \colim \, F_X
$$
where $\mathrm{PSch}$ denotes the category of paraschematic spaces.
\end{definition}

We state a result found in \cite{salas}, though in a different language
with slightly different assumptions (quasi-compact and quasi-separated schemes). The proof is a simple check.

\begin{proposition}\label{prop:Salas}
Let $S$ be a scheme and let $\cU$ be a finite covering of $S$ by affine schemes. Then $S_\cU$ is paraschematic and $\Spec^S(S_\cU)\cong S$ in the category of schemes.
\end{proposition}

We now turn to examine the ringed space $S_\cU^2$ (see Definition \ref{Fund::preschspcover}), which is more interesting to
our purposes.

\begin{proposition}\label{prop:fispcssalascompschemes}
Let $S$ be a scheme and let $\cU$ be a finite covering of $S$ by affine open subschemes. Then $S_\cU^2$ is paraschematic. In addition, $\Spec^S(S_\cU^2)\cong S$ in the category {$\Sch$}.
\end{proposition}
\begin{proof}
First of all, we check that it is paraschematic. {As $S$ is semi-separated, this follow from the very definition of $S_\cU^2$ as we have that} 
for every $p\leq q$ the ring morphisms $\cO_{S_\cU^2}(U_p)\rightarrow\cO_{S_\cU^2}(U_q)$ corresponds, via the antiequivalence between affine schemes and rings, to an open embedding.
Using this equivalence, we see that the {functor} ${F^2_\cU:}P_{G(\cU)_\bullet}\rightarrow \Rings$ (see Definition \ref{Fund::preschspcover}) inducing the functor $$F_X:P_{G(\cU)_\bullet}^{op}\rightarrow \Sch$${(see Definition \ref{def:SpecS})} is isomorphic to $F_{\mathrm{can}}^\cU:\Gamma(G(\cU)_\bullet)^{op}\rightarrow\RngSpcs$. Using Lemma \ref{lemma:Fcan}, this concludes the proof.
\end{proof}

\begin{remark}
By Propositions \ref{prop:Salas}, \ref{prop:fispcssalascompschemes},
starting with a scheme $S$ and a
finite affine cover $\cU$ of it, we can build two distinct paraschematic spaces that we called $S_\cU$ and $S_\cU^2$. The former is a schematic space as in \cite{salas} and has some good cohomological properties (see \cite{salashomcoh}). The latter might not be schematic (see Remark \ref{rmk:Ynotschematic}), but its underlying finite topological space arises from the poset of simplices of a regular 2 dimensional semisimplicial set, while this is not always the case for $S_\cU$. Despite of these differences, in both cases (see \cite{salas} and Proposition \ref{prop:fispcssalascompschemes}) we have that $\Spec^S(S_\cU)\cong\Spec^S(S_\cU^2)\cong S$.
\end{remark}

In our next remark we point out that the paraschematic space $S_\cU^2$ we obtained, while being paraschematic, a priori might not be in general schematic in the sense of \cite{salas}.

\begin{remark}\label{rmk:Ynotschematic}
In \cite{salas} a paraschematic space $(X,\cO_X)$ is said to be \emph{schematic} if for any $p,q\in X$, any $p'\geq p$ and any $i\geq 0$ the natural morphism $H^i(U_p\cap U_q,\cO_X)\otimes_{\cO_X(U_p)}\cO_{X}(U_{p'})\rightarrow H^i(U_{p'}\cap U_p,\cO_X)$ is an isomorphism (see \cite{salas} 4.1 and 4.2).

A paraschematic space $S_\cU^2$ as above is not always schematic. Indeed, assume that in the covering $\cU$ there are four {distinct} affine schemes $U_1$, $U_2$, $U_3$ and $U_4$ such that $U_1\cap U_2=:\Spec(A)$, $\Spec(B):=U_1\cap U_2\cap U_3$, $\Spec(C):=U_1\cap U_2\cap U_4$ and that $\Spec(D):=U_1\cap U_2\cap U_3\cap U_4$ is not empty. The corresponding simplices in $G(\cU)_\bullet$ can be depicted as follows.
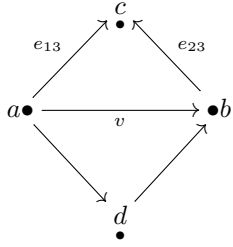
\begin{displaymath}\label{cubecat}
\xymatrix{
 & \underset{\bullet}{c} & \\
 a\bullet \ar[rr]_{v} \ar[ur]^{e_{13}} \ar[dr] & & \bullet b \ar[ul]_{e_{23}}\\
  & \underset{\bullet}{d} \ar[ur] &
  }
\end{displaymath}where $a,b,c,d$ correspond to $U_1,U_2,U_3,U_4$ respectively and $v$ corresponds to $U_1\cap U_2$. We denote as $p,q$ the 2-simplices corresponding to the upper and lower 3-clique in the figure. In this case, we have that $U_a, U_b, U_c,U_d, U_v, U_p$ and $U_q$ seen as open subspaces of $S_\cU^2$ are affine ringed spaces in the sense of Salas \cite[Definition 3.10]{salas} by 4.11 in \cite{salas} (note they are acyclic because of 2.13 of \cite{salas}). Now, if $S_\cU^2$ were schematic, then $U_v$ would be affine and schematic \cite{salas}
and in particular it should be semiseparated \cite[Proposition 4.9]{salas}. However, $\cO_{S_\cU^2}(U_p\cap U_q)\cong \textbf{1}$ is the trivial ring as $U_p\cap U_q=\emptyset$ but $\cO_{S_\cU^2}(U_p)\otimes_{\cO_{S_\cU^2}(U_v)}\cO_{S_\cU^2}(U_q)\cong B\otimes_AC\cong D$ contradicting 4.12 (2) in \cite{salas}. Therefore, $S_\cU^2$ is not schematic.
\end{remark}

We end this section with a remark regarding quasi-coherent modules on (para) schematic spaces.

\begin{remark}\label{rmk:qcohsu2}
Let $S$ be a scheme and $\qCoh(S)$ the category of quasi-coherent modules on $S$ \cite[Chapter 2]{hartshorne}.
For a given finite ringed space $W$, we also have the notion of quasi-coherent module on it and we denote by $\qCoh(W)$ the category of quasi-coherent modules on $W$. Then, by faithfully flat descent \cite{Vis}, we have an equivalence of categories:
$$
\qCoh(S)\cong\qCoh(S_\cU^2, \cO_{S_\cU^2})
$$
We leave to the reader all the checks involved.
\end{remark}

\subsection{The category $\schtwo$}\label{sect:schtwo}
Let be $S$ be a scheme, $\cU$ a finite covering by affine open subschemes of $S$.
Recall we assume all schemes to be quasi-compact and semi-separated.
Moreover, all graphs we consider are finite and belong to either $\diGraphs_{\leq 1}$
or $\Graphs_\leq 1$.

We want to understand some key properties of the finite ringed space $S_\cU^2$ as defined in \ref{Fund::preschspcover}, leading us first to define the category $\schtwo$, its localization and
then to our main result of Sec. \ref{subsec:schemes} Theorem \ref{main-res}. 

The covering $\cU$ determines the 2-dimensional semisimplicial set $G(\cU)_\bullet$, which has
1-skeleton isomorphic to a digraph having at most one edge connecting each pair of distinct vertices (vertices correspond to open sets in $\cU$, while edges to intersections of two open sets).
{Recall that }the finite ringed space $S_\cU^2$ allows us to define
two functors (see Sec. \ref{sect:Gluingringedspaces}):
$$
\begin{array}{c}
F_{\cU}^2: P_{G(\cU)_\bullet}\cong\Gamma(G(\cU)_\bullet)\rightarrow \Rings, \quad
p_i\mapsto \cO_S(U_i),\, p_{ij}\mapsto \cO_S(U_{ij}),\, p_{ijk}\mapsto\cO_S(U_{ijk})\\ \\
F_{\mathrm{can}}^{\cU}: \Gamma(G(\cU)_\bullet)^{op}\rightarrow \RngSpcs, \quad
p_i\mapsto U_i,\, p_{ij}\mapsto U_{ij},\, p_{ijk}\mapsto U_{ijk}
\end{array}
$$
identifying with a small abuse of
notation the objects of $\Gamma(G(\cU))^{op}$ with those in
$P_{G(\cU)_\bullet}\cong\Gamma(G(\cU)_\bullet)$ (see also Obs. \ref{graph-poset} and
\ref{rmk:gbullet}).

Let $T$ be a {{2}}-sub-semisimplicial
set of $G(\cU)_\bullet$ as in Defs. \ref{glcube}, \ref{def:grfin}. Notice
that $\Gamma(T)^{op}=\cube$, so that, by Prop.  \ref{prop:gluingdatsimp} 
$(F_{\mathrm{can}}^{\cU})_{|\Gamma(T)^{op}}$
is a gluing cube, hence it corresponds to a gluing datum (Defs. \ref{gluing-datum},
\ref{gluing-rs}, Sec. \ref{GDrngspcs}).
Using the anti-equivalence
between affine schemes and rings and the fact that $\cU$ consists of affine schemes,
we see that $(F_{\mathrm{can}}^{\cU})_{|\Gamma(T)^{op}}$ being a gluing cube is equivalent to

$F_{\cU}^2$ preserving pushouts, i.e. fibered coproducts. 
Our purpose is to define a category $\schtwo$, having objects constructed using the semisimplicial sets we have introduced so far, that we can localize at a certain class of morphisms to get a category equivalent
to the category of schemes. Heuristically, if $S$ is a scheme and $\cU$ a finite
affine covering, we identify $S$ with the object $(G(\cU)_\bullet, F^2_\cU)$ of $\schtwo$.
However, to define $\schtwo$ and the equivalence of category,
we need a preliminary definition followed by
a discussion.
\begin{definition}\label{def:ext2simp}
Let be $G$ be a graph and $F:\Gamma(G_\bullet)^{op}\rightarrow\RngSpcs$
a gluing functor (Def. \ref{def:gluingfunctor}).
We define the \emph{the degenerate expansion} $\G_\bullet$ of $G_\bullet$ to be the 2-dimensional semisimplicial set obtained by adding to $G_\bullet$ exactly one self loop for each vertex of $G$ and all the possible
two dimensional simplices $u\to u\to u$, $u\to u\to v$, $u\to v\to v$ and  $v\to v\to v$
such that $u, v\in G_0$ and $u\to v$ is an element of $G_1$. Notice that we are neither adding extra vertices,
nor edges, except loops.

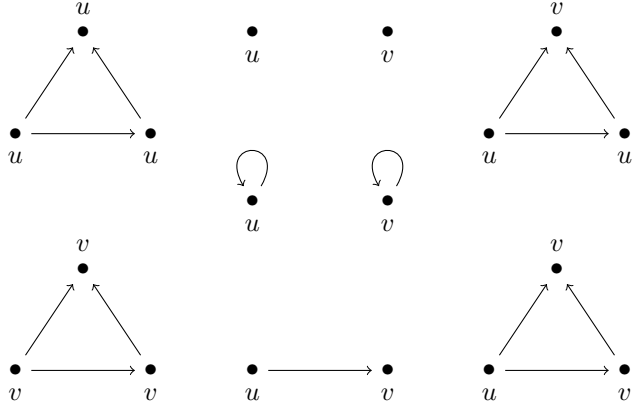
\begin{figure}[h!]
\begin{center}
    \begin{tikzpicture}[scale=.45]
    \node (u) at (-2,0) {$\bullet$};
    \node at (-2,-0.7) {$u$};
    \draw [->] (u) to [out=60,in=120, loop] (u);

    \node (vp) at (-2,5) {$\bullet$};
    \node at (-2,4.3) {$u$};

    \node (up) at (2,5) {$\bullet$};
    \node at (2,4.3) {$v$};
    
    \node (v) at (2,0) {$\bullet$};
    \node at (2,-0.7) {$v$};
    \draw [->] (v) to [out=60,in=120, loop] (v);

    \node (x) at (5,2) {$\bullet$};
    \node (y) at (9,2) {$\bullet$};
    \node (z) at (7,5) {$\bullet$};
    \node at (5,1.3) {$u$};
    \node at (9,1.3) {$u$};
    \node at (7,5.7) {$v$};
    \draw [->] (x) to  (y);
    \draw [->] (y) to  (z);
    \draw [->] (x) to  (z);
    
    \node (y2) at (-5,2) {$\bullet$};
    \node (x2) at (-9,2) {$\bullet$};
    \node (z2) at (-7,5) {$\bullet$};
    \node at (-5,1.3) {$u$};
    \node at (-9,1.3) {$u$};
    \node at (-7,5.7) {$u$};
    \draw [->] (x2) to  (y2);
    \draw [->] (y2) to  (z2);
    \draw [->] (x2) to  (z2);
    
    \node (x3) at (5,-5) {$\bullet$};
    \node (y3) at (9,-5) {$\bullet$};
    \node (z3) at (7,-2) {$\bullet$};
    \node at (5,-5.7) {$u$};
    \node at (9,-5.7) {$v$};
    \node at (7,-1.3) {$v$};
    \draw [->] (x3) to  (y3);
    \draw [->] (y3) to  (z3);
    \draw [->] (x3) to  (z3);

    \node (ua) at (-2,-5) {$\bullet$};
    \node (va) at (2,-5) {$\bullet$};
 
    \node at (-2,-5.7) {$u$};
    \node at (2,-5.7) {$v$};
    \draw [->] (ua) to  (va);

    \node (y4) at (-5,-5) {$\bullet$};
    \node (x4) at (-9,-5) {$\bullet$};
    \node (z4) at (-7,-2) {$\bullet$};
    \node at (-5,-5.7) {$v$};
    \node at (-9,-5.7) {$v$};
    \node at (-7,-1.3) {$v$};
    \draw [->] (x4) to  (y4);
    \draw [->] (y4) to  (z4);
    \draw [->] (x4) to  (z4);
    \end{tikzpicture}
    \end{center}
    \caption{All the simplices of the degenerate expansion $\G_\bullet$ of $
G_\bullet: 
 u\rightarrow v
$}
\end{figure}

We call the added simplices \emph{degenerate simplices} of $\G_\bullet$.
\par
We also define the \emph{the degenerate expansion} $\F$ of $F$ to be the functor
$\F: \Gamma(\G_\bullet)^{op}\rightarrow\RngSpcs$ as follows.

1) {$\F|_{\Gamma(G_\bullet)^{op}}=F$} 

2) Let $u, v\in G_0$ and $u\to v \in G_1$. We set
$$
\begin{array}{cc}
\F(u\to u):=F(u), & \F(u\to u\to u):=F(u),  \\  \F(u\to u\to v):=F(u \to v),
& \F(u\to v\to v):=F(u \to v)
\end{array}
$$

3) For each morphism $u\to (u\to u)$ in $\Gamma(\G_\bullet)$ its image
$\F(u\to u)=F(u)\rightarrow \F(u)=F(u)$ is defined to be the identity, for each morphism $(u\to v)\to(u\to u\to v)$ and $(u\to v)\to(u\to v\to v)$, $\F(u\to u\to v)\to \F(u\to v)$ and $\F(u\to v\to v)\to \F(v\to v)$ are defined to be the identity and for each morphism $(u\to u)\to(u\to u\to v)$ and $(u\to u)\to(v\to u\to u)$, $\F(u\to u\to v)\to \F(u\to u)$ and $\F(v\to u\to u)\to \F(u\to u)$ are defined to be the open embeddings $F(u\to v)\hookrightarrow F(u)$ and $F(v\to u)\hookrightarrow F(u)$ respectively.
\end{definition}

\begin{observation}\label{obs:extendediso}
In the notation of Definition \ref{def:ext2simp}
one can check that the colimit of the functors $F$ and $\F$ are isomorphic.
\end{observation}

We are ready to define the category $\schtwo$ that we shall localize at a certain class of morphisms
and compare to the one
of schemes.

\begin{definition}\label{schtwodef}
We define the category $\schtwo$ as follows.

\noindent
{\it Objects:}
the objects of $\schtwo$ are pairs
$A=(\cA,F_A:\Gamma(\cA)\rightarrow\Rings)$ {where} $\cA$ is a $2$-dimensional finite
semisimplicial set isomorphic to the degenerate expansion $\G_\bullet$ of a given
semisimplicial set $G_\bullet$, $G$ a graph, and
the functor $\Spec\circ F_A$ factors through the degenerate expansion of a
gluing functor $\Gamma(G_\bullet)^{op}\rightarrow\RngSpcs$.

\noindent
{\it Arrows:} the morphisms of $\schtwo$
are pairs $\psi=(f,\varepsilon):A=(\cA,F_A)\rightarrow B=(\cB, F_B)$,
where $f:\cA\rightarrow\cB$ is a morphism between $2$-dimensional finite semisimplicial sets and
$\varepsilon:F_B\circ\Gamma(f)\rightarrow F_A$ is a natural transformation.
\end{definition}

\begin{observation}\label{rmk:schtwowe}
We notice some key facts regarding the category $\schtwo$.
\begin{enumerate}
\item Given an object $(\cA, F_A)$ in $\schtwo$,
the functor $F_A:\Gamma(\cA)\rightarrow\Rings$ gives a sheaf on a base
for the topological space $A(P_{\cA})$. Hence we can associate
to $(\cA, F_A)$ the ringed space $(A(P_{\cA}), \cO_A)$, where $\cO_A$ is obtained by from $F$
{by noticing that $F_A$ is a sheaf of rings on a base}.
This space is also paraschematic, as one can readily see. 
\item Given two objects $A=(\cA, F_A)$, $B=(\cB, F_B)$ in $\schtwo$,
a morphism $\psi=(f,\varepsilon):A\rightarrow B$ gives rise to a map between
the corresponding paraschematic spaces.
Indeed $f$ induces a continuous map $|f|:A(P_{\cA})\rightarrow A(P_{\cB})$.
{From} the natural transformation $\varepsilon$ we get a sheaf morphism
$\bar{\varepsilon}:|f|^{-1}\cO_B\rightarrow \cO_A$ and so, by adjunction,
a morphism $\varepsilon^{\flat}:\cO_B\rightarrow |f|_\ast\cO_A$, see \cite{GW} page 55. 
\end{enumerate}

\end{observation}

Given an object $A=(\cA,F_A)\in\schtwo$, by applying the functor $\Spec^S$ to $(A(P_{\cA}), \cO_A)$,
obtained as in Obs. \ref{rmk:schtwowe}, we get a scheme.
Hence, we can give the following definition.

\begin{definition}\label{specC}
Let the notation be as above. We define the functor:
$$
\Spec^\cC:\schtwo\rightarrow\Sch, \qquad \Spec^\cC(\cA,F):=\Spec^S(A(P_{\cA}), \cO_A),
$$
the definition on the morphisms being clear.
\end{definition}

We conclude this section with a key definition and an observation we will need in
the sequel.

\begin{definition}
Let $S$ be a scheme and $\cU$ a finite covering by open affine subschemes of $S$.
Define the functor
$$
\F^2_\cU:P_{\G(\cU)_\bullet} \cong \Gamma(\G(\cU)_\bullet)\rightarrow\Rings,
$$
via  $\F^\cU_{\mathrm{can}}$, 
the degenerate expansion of  $F^\cU_{\mathrm{can}}$.
Define also:
$$
(\cU)^2_{\mathrm{sch}}:=(\G(\cU)_\bullet,\F^2_\cU)
$$
\end{definition}

\begin{observation}\label{obs:revschtwo}
We have
that $(\cU)^2_{\mathrm{sch}}$
is an object of $\schtwo$ and that $\Spec^\cC((\cU)^2_{\mathrm{sch}})\cong S$
by Observation \ref{obs:extendediso}.
\end{observation}

\subsection{The category $\schtwo$ and weak equivalences}
We want to localize the category $\schtwo$ at a class of morphisms, containing all isomorphisms
and closed under composition, that we call, with a small abuse of terminology, {\it a set of weak equivalences}.

\begin{definition}
Let us define the set of $\cW$ in $\schtwo$
as the set of morphisms $\psi=(f,\varepsilon)$ such that the map $\varepsilon^\flat$
(Obs. \ref{rmk:schtwowe})
is an isomorphism and $\Spec^\cC(\psi)$ is affine. 
As one can readily check $\cW$ is a set of morphisms containing all isomorphisms
and closed under composition, hence $\cW$ is a set of weak equivalences. We shall refer to its elements as {\it weak equivalences}.
\end{definition}
We now record in the following proposition some useful facts concerning the
set of weak equivalences $\cW$.

\begin{proposition}\label{prop:techschtwo1}
The following are true.
\begin{itemize}

\item[1)] The functor $\Spec^\cC$ sends weak equivalences to isomorphisms of schemes. 
\item[2)] If $A=(\cA,F_A)\in\schtwo$ and $\Spec^\cC(\cA,F_A)=: S$, we have that $\cU_A:=\lbrace\Spec(F_A(p))\rbrace_{p\in\cA_0}$ is an affine open cover of $S$ and we have a canonical weak equivalence $i_A=(f,\varepsilon):(\cU_A)^2_{\mathrm{sch}}\hookrightarrow A$ such that $f$ is an inclusion of 2-dimensional semisimplicial sets and $\varepsilon$ is the identity.
\item[3)] Using the notation of 2), let $\cV=\{ V_j\}_{j\in J}$ be a finite covering by open affine subschemes of $S$ that is a refinement of $\cU_A$. Assume that for all $p\in\cA_0$ there exist a subset $J_p\subseteq J$ such that $\{ V_j\}_{j\in J_p}$ is a covering of $\Spec(F_A(p))\subseteq S$ and the $J_p$ are a partition of $J$. Then there exists a canonical weak equivalence $(\cV)^2_{\mathrm{sch}}\rightarrow(\cU_A)^2_{\mathrm{sch}}$ that is an inclusion on the underlying 2-dimensional semisimplicial sets.
\end{itemize}
\end{proposition}
\begin{proof}

1) Consider a weak equivalence $\psi=(f,\varepsilon):A=(\cA,F_A)\rightarrow B=(\cB, F_B)$. Denote  $\Spec^\cC(A)=: X$, $\Spec^\cC(B)=: Y$ and consider the paraschematic spaces $(A(P_\cA),\cO_A)$ and $(A(P_\cB),\cO_B)$ associated to $A$ and $B$ as in Observation \ref{rmk:schtwowe} (2). We want to prove that $\Spec^\cC(\psi):X\rightarrow Y$ is an isomorphism. As isomorphisms are local on target (see \cite{GW} Appendix C for this notion), it suffices to show that for all $v\in\cB_0$ the morphism $\Spec^\cC(\psi)^{-1}(\Spec(\cO_B(U_v))\rightarrow \Spec(\cO_B(U_v))$ is an isomorphism. By definition of weak equivalence, we know that $\Spec^\cC(\psi)$ is affine. As a consequence $\Spec^\cC(\psi)^{-1}(\Spec(\cO_B(U_v))$ is an open affine subscheme of $X$ so $\Spec^\cC(\psi)^{-1}(\Spec(\cO_B(U_v))\rightarrow \Spec(\cO_B(U_v))$ is a morphism of affine schemes and therefore to check that it is an isomorphism it suffices to show that its induced map on global sections is a ring isomorphism. First of all, $\Spec^\cC(\psi)^{-1}(\Spec(\cO_B(U_v))\cong\Spec(\cO_A(|f|^{-1}(U_v)))$. Moreover, as $\psi$ is a weak equivalence, $\varepsilon^{\flat}$ is an isomorphism. Therefore, we get the required isomorphism.\\
We shall now prove 2). By definition, $\cA\cong \G_\bullet$ for some graph $G$ and $\Spec\circ F_A$ factors through the degenerate expansion of a gluing functor
As a consequence $\cU_A$ is a covering by open affine subschemes of $\Spec^\cC(A)$. Here a comment should be made: strictly speaking, a covering by open embeddings of a scheme is a set and therefore, as it could happen that the open embeddings $\Spec(F_A(p))\hookrightarrow S$ and $\Spec(F_A(q))\hookrightarrow S$ are equal (so in particular $\Spec(F_A(p))=\Spec(F_A(q))$), we end up with a covering of S
that a priori does not have an element for each $p\in\cA_0$. This is not substantial as in this case we can for example replace one of the schemes involved with another one isomorphic to it.

Consider now the semisimplicial set $G(\cU_A)_\bullet$. By construction we have a $\varphi:G(\cU_A)_0\xrightarrow{\cong} G_0$. Consider $u,v\in G(\cU_A)_0$ and assume that $\varphi(u),\varphi(v)\in G_0$ are connected by and edge $e\in G_1$ such that $F_A(e)$ is not trivial. Then, as $\Spec\circ F_A$ factors through the degenerate expansion of a gluing functor we have that $u$ and $v$ are connected by an edge as well. Conversely, if $F_A(e)$ is trivial or $\varphi(u)$ and $\varphi(v)$ are disconnected then $u$ and $v$ are disconnected. It follows that $\varphi$ can be extended to an inclusion of graphs $G(\cU_A)\subseteq G$ and therefore $G(\cU_A)_\bullet$ and $\G(\cU_A)_\bullet$ are sub semisimplicial sets of $G_\bullet$ and $\G_\bullet$ respectively. Define $f$ to be the inclusion $\G(\cU_A)_\bullet\subseteq\G_\bullet$ just obtained. For a given $p\in \G(\cU_A)_\bullet$ we define the ring morphism $\varepsilon (p):(F_A\circ\Gamma(f))(p)\rightarrow \cO_S(U_p)$ to be the one induced by the isomorphism $\Spec(F_A(p))$ and the open subscheme $U_p$ corresponding to it in $S$ (note that this definition does not require choices). As $(f,\varepsilon)$ is trivially a weak equivalence this completes the proof of 2).\\We shall now prove 3) by building explicitly a canonical weak equivalence $(r,\varepsilon_r):(\cV)^2_{\mathrm{sch}}\rightarrow(\cU_A)^2_{\mathrm{sch}}$. To ease the notation, let us denote as $\{U_i\subseteq S\}_{i\in I=\cA_0}$ the covering $\cU_A$ and for each $i\in I$ let us denote as $\{V_{ij}\subseteq U_i\}_{j\in J_i}$ the subcover of $U_i$ as in the hypothesis. The fact that $\cV$ is a refinement of $\cU_A$ together with the fact that each $V_{ij}$ is a subscheme of a particular
$U_i$ allows us to define a graph morphism $\tilde{r}:G(\cV)\rightarrow G(\cU_A)$ in the obvious way: $\tilde{r}(V_{ij}):= U_i$ and $\tilde{r}(V_{ij}\cap V_{lk})=U_i\cap U_l$ (here we are denoting with the intersections their corresponding edges if the considered intersections are not empty). This map clearly extends to a map $r:\G(\cV)_\bullet\rightarrow \G(\cU_A)_\bullet$. Finally, we construct the required $\varepsilon_r$ considering the inclusions of affine schemes $V_{ij}\subseteq U_i$, $V_{ij}\cap V_{lk}\subseteq U_i\cap U_l$ and $V_{ij}\cap V_{lk}\cap V_{mn}\subseteq U_i\cap U_l\cap U_m$.
\end{proof}

\begin{definition}
Let $S$ be a scheme and $\cU=\{ U_k\}_{k\in K}$ be a finite cover of it by open but possibly not affine subschemes.
We say that a finite refinement $\cV=\{ V_j\}_{j\in J}$ of $\cU$ by affine schemes is a \emph{complete affine subcover} of $\cU$ if for all $k\in K$ there exists a subset $J_k\subseteq J$ such that $\{ V_j\}_{j\in J_k}$ is a covering of $\Spec(U_k)\subseteq S$ and the $J_k$ are a partition of $J$.
\end{definition}

\begin{definition}\label{cor:techschtwo2}
Let $S$ be a scheme and $\cV=\{V_j\subseteq S\}_{j\in J}$, $\cU=\{U_i\subseteq S\}_{i\in I}$ two finite open affine coverings of it. We define the covering $\cV\times_S\cU$ of $S$ as 
$$\cV\times_S\cU:=\{U_i\times_S V_j\subseteq S\}_{(i,j)\in I\times J}$$
\end{definition}

\begin{observation}
Notice that, as $S$ is semiseparated, the covering $\cV\times_S\cU$ is a complete affine subcover of both $\cU$ and $\cV$. Therefore, mimicking the proof of Proposition \ref{prop:techschtwo1} 3) we can define weak equivalences $p_{\cV}:(\cV\times_S\cU)^2_{\mathrm{sch}}\rightarrow (\cV)^2_{\mathrm{sch}}$, $p_{\cU}:(\cV\times_S\cU)^2_{\mathrm{sch}}\rightarrow (\cU)^2_{\mathrm{sch}}$.
\end{observation}

\subsection{Schematic Localization of $\schtwo$}
We would like to localize the category $\schtwo$ at weak equivalences.
One of the classical ways to do this is to prove that the given
class of weak equivalences is either a
right or a left multiplicative system. This is not how we will proceed,
but, since we modify such construction, let us briefly
recall the definition of right multiplicative system,
see \cite[\href{https://stacks.math.columbia.edu/tag/04VC}{Tag 04VC}]{SP}.

\begin{definition}\label{def:rms}
Let $\cC$ be a category. A set of arrows $\cS$ of $\cC$ is called a
\textit{right multiplicative system} if it has the following properties:
\begin{itemize}
\item [RMS1)] The identity of every object of $\cC$ is in $\cS$ and the composition of two composable elements of $\cS$ is in $\cS$.
\item [RMS2)] For every solid diagram
$$
\xymatrix{X\ar@{.>}[r] \ar@{.>}[d]_t & Y\ar[d]^s\\
Z\ar[r] & W\\
}
$$
with $s \in \cS$, there exists an $X$ (not unique in general), so that
the diagram can be completed to a commutative dotted square with $t \in \cS$.
\item [RMS3)] For every pair of morphisms $f,g:X \lra Y$ and $s\in \cS$ with source $Y$ such that $s \circ f=s \circ g$  there exists a $t \in \cS$ with target $X$, such that $f\circ t=g \circ t$.
\end{itemize}
\end{definition}

When the set of weak equivalences is a right (or left) multiplicative system, the usual calculus of fractions can be used to construct an explicit model of the desired localization (see \cite[\href{https://stacks.math.columbia.edu/tag/04VB}{Tag 04VB}]{SP} or the very nice exposition in \cite{Borceux}).

In our case, however, the set of weak equivalences $\cW$ fails to be a multiplicative system as we see in the next example, showing an explicit counterexample to property RMS3. 

\begin{example}
Let be $R_X$ a commutative ring and let $X:=\Spec(R_X)$. Consider objects $A=(\cA,F_A)$ and $B=(\cB,F_B)$ where $\cA$ and $\cB$ are the degenerate expansions of $G_\bullet$ and $H_\bullet$ where $G$ is the graph consisting of a single vertex $u$ and $H$ is a graph $a\rightarrow b$ consisting of two vertices joined by a single edge. $F_A$ and $F_B$ are defined to be the constant functors with value $R_X$ on the category of simplices of $\cA$ and $\cB$ respectively. Note that $\Spec^\cC(A)$ and $\Spec^\cC(B)$ are both isomorphic to $X$. Consider the two morphisms $\psi_a=(f_a, \varepsilon_a),\psi_b=(f_b, \varepsilon_b):A\rightarrow B$ having $f_a$ and $f_b$ defined extending the unique inclusions sending $u$ to $a$ and $b$ respectively and where $\varepsilon_a$ and $\varepsilon_b$ are obtained using the identity $R_X\rightarrow R_X$.  Consider the map $\pi:(f_\pi,\varepsilon_\pi):B\rightarrow A$, where $f_\pi$ is the semisimplicial set map collapsing the two distinct vertices of $\cB$ to the single vertex of $u$ and $\varepsilon_\pi$ is defined using the identity of $R_X$. $\pi$ is a weak equivalence and $\pi\circ\psi_a=\pi\circ\psi_b$
but it is not possible to find a semisimplicial set $\mathcal{Q}$
and a morphisms $f_s:\mathcal{Q}\rightarrow\G(\cU)_\bullet$ such that $f_a\circ f_s=f_b\circ f_s$, so we get that that RMS3 is not satisfied by the collection of weak equivalences. However, one can check that $\Spec^\cC(\psi_a\circ\pi)=\Spec^\cC(\psi_b\circ\pi)$. As $\Spec^\cC(\pi\circ\psi_a)=\Spec^\cC(\pi\circ\psi_b)$,
we then observe that in our example a modified version of property RMS3 holds at the level of the scheme
morphisms induced by morphisms in $\schtwo$. This suggests a new definition
of right multiplicative system, with a weaker form of RMS3, that
as we shall see, allows for localization.
\end{example}

The previous example motivates the following definitions.

\begin{definition}
We say that two morphisms $\psi,\psi':A\rightarrow B$ in $\schtwo$ are \emph{schematic equal} if $\Spec^\cC(\psi)=\Spec^\cC(\psi')$. We say that a diagram in $\schtwo$ is \emph{schematic commutative} if all the morphisms appearing in it, having the same source and target, are schematic equal.
\end{definition}

By definition, a commutative diagram in $\schtwo$ is also schematic commutative. We will need some useful properties of schematic equal morphisms. We start with a technical lemma.

\begin{lemma}\label{lemma:factortech}
Every morphism $\psi=(f,\varepsilon):A=(\cA,F_A)\rightarrow B=(\cB,F_B)$ in $\schtwo$ restricts to a morphism $\psi'=(f',\varepsilon'):(\cU_A)^2_{\mathrm{sch}}\rightarrow (\cU_B)^2_{\mathrm{sch}}$ such that $\psi\circ i_A=i_B\circ\psi'$.
\end{lemma}
\begin{proof}
Note that if for a one dimensional simplex $e\in\cA$ we have that $F_A(e)$ is not trivial (i.e. $e$ belongs to the subsemisimplicial set $\G(\cU_A)_\bullet\subseteq\cA$), then $F_B(f(e))$ must be not trivial as well (i.e. $f(e)$ belongs to the subsemisimplicial set $\G(\cU_B)_\bullet\subseteq\cB$), as there can be no ring homomorphisms from the trivial ring to a non trivial ring. The same holds true for 2 dimensional simplices. As a consequence, $f$ restricts to a map $\G(\cU_A)_\bullet\rightarrow\G(\cU_B)_\bullet$ and, as a consequence, so is $\varepsilon$.
\end{proof}
\begin{remark}\label{obs:scheqwelldef}
Explicitly note that if $\psi=(f,\varepsilon_f),\psi'=(g,\varepsilon_g):A\rightrightarrows B$ are schematic equal then there exists an isomorphism $\varphi:|f|_\ast\cO_A\rightarrow |g|_\ast\cO_A$ such that $\varepsilon_g^\flat=\varphi\circ\varepsilon_f^\flat$. In addition, if $A$ and $B$ are of the form $(\cU)^2_{\mathrm{sch}}$ and $(\cV)^2_{\mathrm{sch}}$ for two given open affine coverings of two schemes $X$ and $Y$ we can take $\varphi$ to be the identity.
\end{remark}

We now prove an important technical proposition that can be interpreted as the fact that our weak equivalences satisfy a modified version of RMS2 and RMS3 where all the equalities required on the morphisms by these two conditions are relaxed to schematic equalities.

\begin{lemma}\label{prop:RMS2}
The following holds true:
\begin{itemize}
\item[1)]Let be $f:A\rightarrow B$ a weak equivalence and let be $g:C\rightarrow B$ a morphism of $\schtwo$. Then there exist $Q\in\schtwo$, a weak equivalence $f':Q\rightarrow C$ and a morphism $g':Q\rightarrow A$ such that $g\circ f'$ and $f\circ g'$ are schematic equal. In addition, if $g$ is a weak equivalence we can construct $g'$ in such a way that $g'$ is a weak equivalence as well.
\item[2)] Given two maps $f,g:A\rightrightarrows B$, if there exists a weak equivalence $s:B\rightarrow Q$ in $\schtwo$ such that $s\circ f$ is schematic equal to $s\circ g$ then there exists a weak equivalence $t:Q\rightarrow A$ such that $f\circ t$ and $g\circ t$ are schematic equal.
\end{itemize}
\end{lemma}
\begin{proof}
We begin with the proof of 1). Denote $Y:=\Spec^\cC(B)$, $X:=\Spec^\cC(A)$, $Z:=\Spec^\cC(C)$. Because of Lemma \ref{lemma:factortech} we can assume that $A=(\cU_A)^2_{\mathrm{sch}}=:(\cA,F_A)$, $B=(\cU_B)^2_{\mathrm{sch}}=:(\cB,F_B)$ and $C=(\cU_C)^2_{\mathrm{sch}}=:(\cC,F_C)$. As $f$ is a weak equivalence, we know that $\Spec^\cC(f)$ is an isomorphism. Accordingly, $\Spec^\cC(f)(\cU_A)=\{\Spec^\cC(\Spec(F_A(q))\}_{q\in\cA_0}$ is a covering of $Y$ by open affine subschemes. As a consequence, $\cU_{g^{-1}\cU_A}:=\{\Spec^\cC(g)^{-1}(\Spec^\cC(\Spec(F_A(q)))\}_{q\in\cA_0}$ is an open (possibly non affine) covering of $Z$: let $\cV^{\mathrm{aff}}$ be a complete affine subcover of it. We define $Q:=(\cQ,F_Q):=(\cU_C\times_Z\cV^{\mathrm{aff}})^2_{\mathrm{sch}}$ and $f':=p_{\cU_C}$. By Def. \ref{cor:techschtwo2} $f'$ is a weak equivalence. We now need to define $g'=:(\tilde{g}',\varepsilon_{g'})$. We define $\tilde{g}'$ as the composition of the canonical maps $\G(\cU_C\times_Z\cV^{\mathrm{aff}})_\bullet\to\G(\cV^{\mathrm{aff}})_\bullet\to\G(\cU_{g^{-1}\cU_A})_\bullet\cong\G(\cU_A)_\bullet=\cA$. We define $\varepsilon_{g'}$ as follows. For every $a\in\Gamma(\G(\cU_C\times_Z\cV^{\mathrm{aff}})_\bullet)$, we define the map $\varepsilon_{g'}(a):F_A(\tilde{g}'(a))\rightarrow F_Q(a)$ as the ring homomorphism given by the global sections of the scheme morphism obtained composing the open embedding $\Spec(F_Q(a))\rightarrow \Spec^\cC(g)^{-1}(\Spec^\cC(\Spec(F_A(\tilde{g}'(a))))$ (recall that $\cV^{\mathrm{aff}}$ is a complete affine subcover of $\cU_{g^{-1}\cU_A}$) with the map $$\Spec^\cC(g)^{-1}(\Spec^\cC(\Spec(F_A(\tilde{g}'(a))))\rightarrow \Spec^\cC(\Spec(F_A(\tilde{g}'(a)))$$ obtained by taking the pullback of $\Spec^\cC(g)$ along the morphism $$\Spec(F_A(\tilde{g}'(a)))\hookrightarrow X\xrightarrow{f} Y$$ One can check that $g'$ is the required morphism and that it is a weak equivalence if $g$ is a weak equivalence.\\
We turn to the proof of 2). Denote $Y:=\Spec^\cC(B)$, $X:=\Spec^\cC(A)$, $Z:=\Spec^\cC(Q)$. Again, we can assume that $A=(\cU_A)^2_{\mathrm{sch}}=:(\cA,F_A)$ and that $B=(\cU_B)^2_{\mathrm{sch}}=:(\cB,F_B)$. Consider the covering by open (possibly non affine) subschemes $\cU_f:=\Spec^\cC(f)^{-1}(\cU_B)$ and $\cU_g:=\Spec^\cC(g)^{-1}(\cU_B)$ of $X$. As $s$ is a weak equivalence, $\Spec^\cC(s)$ is an isomorphism. The hypothesis implies $\Spec^\cC(s)\circ\Spec^\cC(f)=\Spec^\cC(s)\circ\Spec^\cC(g)$ so that $\Spec^\cC(s)=\Spec^\cC(f)$. As a consequence $\cU_f=\cU_g$: consider a complete affine subcover $\cV$ of it. Define $Q:=(\cU_A\times_X\cV)^2_{\mathrm{sch}}$ and $t:=p_{\cU_A}$. One can check that these are the desired object and morphism.
\end{proof}

\begin{definition}\label{def:srms}
A set of arrows $\cS$ of $\schtwo$ is called a \textit{schematic
right multiplicative system}, if it has the following properties:
\begin{itemize}
\item [sRMS1)] The identity of every object of $\cC$ is in $\cS$ and the composition of two composable elements of $\cS$ is in $\cS$.
\item [sRMS2)] For every solid diagram
$$
\xymatrix{X\ar@{.>}[r]^{t'} \ar@{.>}[d]_t & Y\ar[d]^s\\
Z\ar[r]^{s'} & W\\
}
$$
with $s \in \cS$, there exists an $X$ (not unique in general), so that
the diagram can be completed to a schematic commutative diagram with $t \in \cS$. If $s'\in\cS$ then one can complete the diagram in such a way that $t'\in\cS$ as well.
\item [sRMS3)] For every pair of morphisms $f,g:X \lra Y$ and $s\in \cS$ with source $Y$ such that $s \circ f$ is schematically equal to $ s \circ g$  there exists a $t \in \cS$ with target $X$, such that $f\circ t$ is schematically equal to $g \circ t$.
\end{itemize}
\end{definition}

By Lemma \ref{prop:RMS2} we have immediately the following result, which is the key for the localization of $\schtwo$ at $\cW$.

\begin{proposition}
The set $\cW$ of weak equivalences in $\schtwo$ is a schematic right multiplicative system.
\end{proposition}

We are now ready to define the localized category $\schtwo[\cW^{-1}]$ mimicking the usual construction used in the case of localization at right multiplicative system (see for example \cite[5.2.4]{Borceux} or \cite{salas} for a construction closer to our case).

\begin{definition}
We define the \textit{schematic localization}  $\schtwo[\cW^{-1}]$ of $\schtwo$ at the schematic multiplicative system $\cW$ to be the
the category defined as follows:
\begin{itemize}
\item[1)] The objects of $\schtwo[\cW^{-1}]$ coincide with the objects of $\schtwo$.
\item[2)] The morphisms from an object $A$ to an object $B$ are obtained considering the set of zig-zags $A\xleftarrow{s} P\xrightarrow{f} B$, where $s$ is a weak equivalence and $f$ is a morphism of $\schtwo$, and quotienting with the equivalence relation given by declaring two zig-zags $A\xleftarrow{s} P\xrightarrow{f} B$, $A\xleftarrow{s'} P'\xrightarrow{f'} B$ to be equivalent if there exists a zig-zag $P\xleftarrow{u} P''\xrightarrow{v} P'$ such that $u$ and $v$ are weak equivalences, $s\circ u$ and $s'\circ v$ are schematic equal and $f\circ u$ and $f'\circ v$ are schematic equal (see diagram \ref{loc:eqrel}).

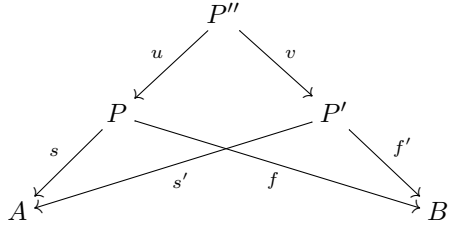
\begin{equation}\label{loc:eqrel}
\xymatrix{ & & P''\ar[dl]_u\ar[dr]^v & & \\
&P\ar[dl]_s\ar[drrr]_f &  & P'\ar[dlll]^{s'}\ar[dr]^{f'} & \\
A & & &  & B
}
\end{equation}

\item[3)] The composition of two zig-zags $A\xleftarrow{s} I\xrightarrow{f} B$ and $B\xleftarrow{t} J\xrightarrow{g} C$ is defined to be any zig-zag $A\xleftarrow{s\circ r} K\xrightarrow{g\circ h} C$ where $K\xrightarrow{r}I$ is a weak equivalence and $K\xrightarrow{h}J$ is a morphism such that $f\circ r$ and $t\circ h$ are schematic equal (see diagram \ref{loc:comp}).
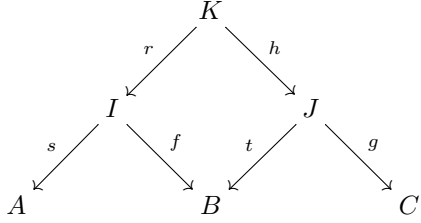
\begin{equation}\label{loc:comp}
\xymatrix{ & & K\ar[dl]_r\ar[dr]^h & & \\
&I\ar[dl]_s\ar[dr]^f &  & J\ar[dl]_{t}\ar[dr]^{g} & \\
A & & B &  & C
}
\end{equation}
\end{itemize}
\end{definition}

\begin{observation}
We explicitly note that the category $\schtwo[\cW^{-1}]$ is not a localization of categories in the usual sense, i.e. a category satisfying a universal property such as the one stated in \cite[5.2.1]{Borceux}. In our case, only the properties  stated in Proposition \ref{prop:univloc} hold. 
\end{observation}

\begin{proposition}
$\schtwo[\cW^{-1}]$ is well defined.
\end{proposition} 
\begin{proof}
The proof follows arguing formally as in the case of the localization of a category at a right multiplicative collection of morphisms (see for example \cite{Borceux} 5.2.4 and the diagrams therein) using Properties sRMS2 and sRMS3 in place of the properties RMS2 and RMS3 of a right multiplicative system (see Definitions \ref{def:srms} and \ref{def:rms}). Note that we require both $u$ and $v$ in point 2) of the previous definition to be weak equivalences (as done also in \cite{salas}, for example), while when localizing at right calculi of fractions usually a weaker hypothesis is required (indeed, in the case of sRMS2, using the terminology of Proposition \ref{prop:RMS2} 1), if $g$ is a weak equivalence it follows that we can take $g'$ to be a weak equivalence as well).
\end{proof}

\subsection{The equivalence of categories between $\schtwo[\cW^{-1}]$ and schemes}
To establish an equivalence of categories between $\schtwo[\cW^{-1}]$ and the category of schemes we first define a functor between them.

\begin{proposition}\label{prop:univloc}
We have the following:
\begin{itemize}
\item[1)] There exists a functor
$$
I_\cW:\schtwo\rightarrow \schtwo[\cW^{-1}]
$$
that is the identity on the objects and which sends a morphism $f:A\rightarrow B$ in $\schtwo$ to the (class of) the zig-zag $A\xleftarrow{id_A}A\xrightarrow{f} B$.
\item[2)]The functor $I_\cW$ factors through a functor
$$
\Spec_\cW^\cC:\schtwo[\cW^{-1}]\rightarrow\Sch
$$
such that $\Spec_\cW^\cC\circ I_\cW=\Spec^\cC$. This functor is the identity on the objects and sends the (class of) a zig-zag $A\xleftarrow{s}P\xrightarrow{f}B$ to $\Spec^\cC(f)\circ\Spec^\cC(s)^{-1}:\Spec^\cC(A)\rightarrow\Spec^\cC(B)$.
\end{itemize}
\end{proposition}
\begin{proof}
The proof follows mutatis mutandis arguing as in \cite[\href{https://stacks.math.columbia.edu/tag/04VK}{Tag 04VK}]{SP} (or as in \cite[page 186]{Borceux}) using the fact that weak equivalences are sent to isomorphisms by the functor $\Spec^\cC$.
\end{proof}

\begin{observation}\label{rmk:essurj}
    Let $S$ be a scheme and $\cU$ a finite covering of it by affine open subschemes. Then one can check that $\Spec_\cW^\cC ((\cU)^2_{\textrm{sch}})\cong S$. 
\end{observation}

\begin{theorem}\label{main-res}
The functor $$\Spec_\cW^\cC:\schtwo[\cW^{-1}]\rightarrow\Sch$$ is an equivalence of categories,
where $\Sch$ is the category of quasi-compact semi-separated schemes.
\end{theorem}
\begin{proof}
We want to prove that $\Spec_\cW^\cC$ is essentially surjective, faithful and full. It is essentially surjective because of Observation \ref{rmk:essurj}.\\We prove that the functor is faithful. To do this, we need to show that if two zig-zags $F:=A\xleftarrow{s}I\xrightarrow{f}B$ and $G:=A\xleftarrow{s'}J\xrightarrow{f}B$ are sent to the same morphism by $\Spec_\cW^\cC$, then they are equivalent. We can assume that $A=(\cU_A)^2_{\mathrm{sch}}$, $B=(\cU_B)^2_{\mathrm{sch}}$, $I=(\cU_I)^2_{\mathrm{sch}}$ and $J=(\cU_J)^2_{\mathrm{sch}}$. As $\Spec^\cC(s)$ and $\Spec^\cC(s')$ are isomorphisms, we have that the covers $\cU^A_I:=\Spec^\cC(s)(\cU_I)$ and $\cU^A_J:=\Spec^\cC(s)(\cU_J)$ are coverings by open affine subschemes of $X:=\Spec^\cC(A)$. We explicitly note that $\G(\cU_I^A)_\bullet\cong\G(\cU_I)_\bullet$ and $\G(\cU_J^A)_\bullet\cong\G(\cU_J)_\bullet$. As a consequence, using the global sections of the affine scheme isomorphisms $$\Spec^\cC(\F_{\cU_A}^2(s'(p)))\supseteq \Spec^\cC(s')(\Spec^\cC(\F_{\cU_J}^2(p)))\xrightarrow{\cong}\Spec^\cC(\F_{\cU_J}^2(p)),\ p\in\G(\cU_J)_\bullet$$ we can define a map $s'_A:(\cU_J^A)^2_{\mathrm{sch}}\xrightarrow{\cong}J$. Analogously, we can construct a map $s_A:(\cU_I^A)^2_{\mathrm{sch}}\xrightarrow{\cong}I$. Consider now the covering $\cV:=\cU_A\times_X\cU_J^A\times_X\cU^A_I$ and the projections $p_I:=p_{\cU^A_I},p_J:=p_{\cU^A_J}$ and $p_A:=p_{\cU_A}$ defined in Def. \ref{cor:techschtwo2}. Let $Q:=(\cV)^2_{\mathrm{sch}}$.
Then the zig-zag $I\xleftarrow{p_I}Q\xrightarrow{p_J}J$ exhibits an equivalence between $F$ and $G$. Indeed, using the very definition of the objects and the morphisms we constructed we can check that $s\circ p_I$ and $s'\circ p_J$ are schematic equal and that $f\circ p_I$ and $g\circ p_J$ are schematic equal.\\
We now prove that $\Spec_\cW^\cC$ is full. Consider two objects $A,B\in\schtwo[\cW^{-1}]$ and any map $f:X:=\Spec^\cC(A)\rightarrow \Spec^\cC(B):=Y$ between them. We need to prove that there exists a zigzag $F:=A\xleftarrow{s}Q\xrightarrow{f'}B$ such that $\Spec_\cW^\cC(F)=f$. Because of Lemma \ref{lemma:factortech} we can assume that $A=(\cU_A)^2_{\mathrm{sch}}$, $B=(\cU_B)^2_{\mathrm{sch}}$. Consider the covering $f^{-1}(\cU_B)$ by open (possibly non affine) subschemes of $X$. Take a complete affine cover $\cQ$ of it, consider the covering $\cV:=\cU_A\times_X\cQ$ (that is still a complete affine subcovering of $f^{-1}(\cU_B)$) and define $Q:=(\cV)^2_{\mathrm{sch}}$, $s:=p_{\cU_A}: Q\xrightarrow{\simeq}A$. As $\cV$ is an open affine subcover of $f^{-1}(\cU_B)$, we have an induced semisimplicial morphism $\tilde{f}:\G(\cV)_\bullet\rightarrow\G(f^{-1}(\cU_B))_\bullet\cong\G(\cU_B)_\bullet$. Using the fact that the functor $\Hom_{\Sch}(-,\Spec^\cC(B))$ is a sheaf \cite[Proposition 3.5]{GW} we obtain, for all $p\in\Gamma(\G(\cV)_\bullet)$ a scheme morphism between affine schemes $X\supseteq\Spec(\F^2_{\cV}(p))\rightarrow\Spec(\F_{\cU_B}^2(\tilde{f}(p)))\subseteq Y$ factoring through an element of $f^{-1}(\cU_B)$ and gluing to $f$. We can use these morphisms, taking global sections, to define a map $\varepsilon_f:\F^2_{\cU_B}\circ\Gamma(\tilde{f})\rightarrow \F^2_{\cV}$. We define $f':=(\tilde{f},\varepsilon_f)$. The zig-zag $A\xleftarrow{s}Q\xrightarrow{f'}B$ is the one needed to complete the proof as one can check that $\Spec^\cC(f')=f\circ\Spec^\cC(s)$.
\end{proof}

\subsection{Finite spaces and manifolds} \label{sec-mflds}
In this section we make some remarks regarding the category of differentiable manifolds.

\medskip
Let $M$ be a differentiable manifold and let us consider, with a small abuse of notation, $M$ also as the ringed space $(M,C^\infty_M)$, where $C^\infty_M$ is the sheaf of differentiable functions on $M$.
The algebra of global sections $C^\infty_M(M)$ allows us to recover the points of $M$ via the {\sl Milnor exercise}
(see \cite{kms,nestruev}). More precisely, we say that a maximal ideal $\mathfrak{m}\subseteq C^\infty_M(M)$
is \textit{real} if $C^\infty_M(M)/\mathfrak{m} \cong \mathbb{R}$. The set of real maximal ideals $\mathrm{Spec}_m(C^\infty_M(M))$
is called the \textit{real spectrum} of $C^\infty_M(M)$ and we have the following natural correspondence (Milnor's exercise):
\begin{equation}\label{milnor}
\mathrm{Hom}(C^\infty_M(M),\R) \cong M
\end{equation}
resembling the correspondence between $\C$-points of a complex affine scheme $(S,\cO_S)$ and the maximal spectrum
of $\cO_S(S)$, the global sections of $\cO_S$.
The real spectrum of a differentiable manifold can then be given a topology, so that the correspondence (\ref{milnor}) is an homeomorphism.
We have even more similarities between the category of smooth manifolds (that we shall assume from now on to have a finite atlas) $\Man$ and affine schemes (see \cite{GS}). In fact:
\begin{itemize}
\item There is fully faithful contravariant functor $\Man\rightarrow \ralg$, $M\mapsto\cCinf_M$.
\item We can define a functor 
$\Specr:\ralg\rightarrow\RngSpcs$ 
such that $\Specr(A)$ $\cong M$, for
$A\cong\cCinf_M$ for some manifold $M$.
\end{itemize}

However, as pointed out in \cite{GS}, the category of smooth manifolds lacks some desirable properties that the category of schemes has. Hence to obtain results similar to the ones of \cite{salas} and their generalization in  Section \ref{subsec:schemes} one should then move to the category of \textit{differentiable spaces} as in \cite{GS}. Without getting into the technicalities of such construction is worth noticing that
most of the treatment in Sec. \ref{fin-sec} can be extended with no difficulty to the case of smooth manifolds and differentiable spaces, in particular:
\begin{itemize}
\item The concept of gluing data and gluing cube in Sec. \ref{GDrngspcs}.
\item Given a smooth manifold $S$ and an open covering $\cU$, the construction of the two ringed spaces $S_\cU$,
$S_\cU^2$ and the functor $F_\cU^2$ in Sec. \ref{sect:Gluingringedspaces}.
\end{itemize}

As for the category $\mantwo$, the construction is more involved, we make some comments below.

\begin{remark}
We could define a category $\mantwo$ by repeating verbatim the construction of the category $\schtwo$ and getting results analogue to the ones of Section \ref{subsec:schemes} using differentiable algebras as in \cite[page 30]{GS} in place of $\Rings$ and the functor $\Specr$ of \cite[page 44]{GS} in place of $\Spec$. Given a manifold $M$ admitting a finite atlas $\cU$, one could then repeating the reasoning of Observation \ref{rmk:essurj}
to get an object of $\mantwo$. We believe objects of the category $\mantwo$ to be useful for applications: we leave the study of this category, and more in general the differentiable manifolds and spaces case, to future work.
Similarly, we observe that everything we have stated for differentiable manifold could
be extended to their super analogues \cite{fi}.
\end{remark}
We end this section with an observation, which is important for the study
vector bundles on graphs in \cite{FSZ26}.

\begin{observation}
Given a differentiable (or complex) manifold $M$, there exists a natural correspondence
between isomorphism classes of $\cCinf$ (or holomorphic) vector bundles equipped with a flat connection and isomorphism classes of locally constant sheaves of real (complex) vector spaces on $M$ (see for example \cite{Voisin} Section 9.2). In particular, if $M$ is a differentiable manifold and $(V,\nabla)$ is a vector bundle on it with a flat connection, the desired locally constant sheaf $L_V$ of real vector bundles on $M$ in this correspondence is $\mathrm{ker}(\nabla)$ (interpreted as the local system of the flat sections of $V$, that is the ones annihilated by $\nabla$) and we have $V\cong L_V\otimes\cO_M$. Now, fix a differentiable manifold $M$, a vector bundle together with a flat connection $(V,\nabla)$ and assume that there exists a finite open cover $\cU=\lbrace U_i\subseteq M\rbrace$ of $M$ trivializing $V$ and such that $L_V$ obtained as before is isomorphic to a constant sheaf of vector spaces when restricted to each open submanifold of the cover $\cU$, that is $L_{V|U_i}\cong \underline{\mathbb{R}}^n$ for all $i$ and a given $n\in\N$. Consider now the space $(M_\cU^2,\cO_{M_\cU^2})$ considered, for example,
in Definition \ref{Fund::preschspcover}.
\begin{itemize}
    \item $V\cong L_V\otimes\cO_M$ defines a locally free module $V_{M_\cU^2}\otimes\cO_{M_\cU^2}$ where $V_{M_\cU^2}$ is a sheaf of $\mathbb{R}$-vector spaces on $M_\cU^2$.
\item $V_{M_\cU^2}$ is equivalent to the datum of a presheaf of vector spaces on the poset $P_{G(\cU)_\bullet}$ where all the restriction morphisms are isomorphisms. The latter combinatorial datum amounts to a choice of an $n$-dimensional vector space $F_v$ for each vertex $v$ of the
graph $G(\cU)$, of an isomorphism $F_v\cong F_w$ for each edge joining two vertices $v$ and $w$ of $G(\cU)$ and some cocycle conditions coming from the combinatorics triple intersections of the elements of the cover $\cU$ (there are no such cocycle conditions if we have a cover with empty triple intersections).
\end{itemize}
As a consequence, our treatment of vector bundles on graphs in \cite{FSZ26} inspired by the work \cite{gaybalmaz} appears to be, under suitable assumptions, the discrete analogous of the study of vector bundles with a flat connection (see also \cite{dimakis}).
\end{observation}

\subsection{Homology and Cohomology}\label{section:homcoh}

We discuss the homology and cohomology of the two finite
ringed spaces examined in Subsection \ref{subsec:schemesrs} $(S_\cU,\cO_{S_\cU})$ and $(S_\cU^2,\cO_{S_\cU^2})$.

Assume that $X$ is a finite topological space and that $\cF$
is a sheaf of abelian groups on it. Then giving a sheaf of abelian groups on $X$ amounts to giving a presheaf of abelian groups on the poset associated to $X$ by Proposition \ref{prop:gross}. 
Recall that this amounts to the datum of an abelian group $\cF(U_p)$ for each $p\in X$ together with morphisms $\cF(U_p)\rightarrow \cF(U_p)$ for each $p\leq q$. Following \cite[1.6.1]{salashomcoh}  we define sheaves $C^i\cF$ and cosheaves $C_i\cF$ on $X$ as $$C^i\cF(U)=\prod_{(x_0<...<x_i)\in U}\cF(U_{x_i})\qquad C_i\cF(U)=\bigoplus_{(x_0<...<x_i)\in U}\cF(U_{x_i})\quad U\subseteq X$$ 

These sheaves bundle to give a complex of sheaves: 
$$C^\bullet\quad 0\rightarrow C^0\cF\rightarrow C^1\cF\rightarrow ...\rightarrow C^n\cF\rightarrow 0\quad n=\mathrm{dim}X$$
Analogously, if $\widetilde{\cF}$ is a cosheaf, that is a sheaf on the topological space obtained from $X$ by considering the dual topology (open sets are the closed sets of $X$), we get a complex of cosheaves:
$$C_\bullet\quad 0\rightarrow C_n\widetilde{\cF}\rightarrow C_{n-1}\widetilde{\cF}\rightarrow ...\rightarrow C_0\widetilde{\cF}\rightarrow 0\quad n=\mathrm{dim}X$$We have the following.

\begin{theorem}[\cite{salashomcoh} 1.6.2, 1.6.4] If $\cF$ and $\widetilde{\cF}$ are a sheaf and a cosheaf of abelian groups on a finite topological space $X$, for any open subset $U\subseteq X$ and closed subset $Z\subseteq X$ we have for all natural numbers $i$ $$H^i(C^\bullet \cF(U))\cong H^i(U,\cF)\qquad H_i(C_\bullet \widetilde{\cF}(Z))\cong H_i(Z,\widetilde{\cF})$$
where $H^i(U,\cF)$, $H_i(Z,\widetilde{\cF})$ are the
ordinary sheaf cohomology and homology groups.
\end{theorem}

Let $S$ be either a quasi-compact and semi-separated scheme or a differentiable manifold, $\cU$ a finite (affine if $S$ is a scheme) covering by open embeddings of $S$.
To this datum we can associate paraschematic spaces $(S_\cU,\cO_{S_\cU})$ and $(S_\cU^2,\cO_{S_\cU^2})$.
Let $M$ be a quasi-coherent module on $S$.
We denote as $\pi_\ast M$ the quasi-coherent $\cO_{S_\cU}$-module obtained by pushforward from $M$ along the canonical projection $\pi:S\rightarrow S_\cU$ and as $\widetilde{M}$ the quasi-coherent $\cO_{S_\cU^2}$ module obtained from the equivalence given by Remark \ref{rmk:qcohsu2}.

\begin{proposition} Let the notation be as above.
Then
$$
H^i(C^\bullet \pi_\ast M(S_\cU))\cong H^i(S,M)\cong H^i(C^\bullet \widetilde{M}(S_\cU^2)), \qquad i=0,1
$$
\end{proposition}
\begin{proof}
The first isomorphism is proved in \cite{salas} while the second one follows from noticing that $C^\bullet \widetilde{M}(S_\cU^2))$ coincides in degrees $0,1,2$ with the \v{C}ech complex of $M$ with respect to the cover $\cU$.
\end{proof}

Notice that for $i\geq 2$, in general we have that
$H^i(C^\bullet \pi_\ast M(S_\cU))$ $\not\cong$ $H^i(C^\bullet \widetilde{M}(S_\cU^2))$,
see Rem. \ref{rmk:Ynotschematic}.

\end{document}